\documentclass[12pt]{amsart}
\usepackage{comment}
\usepackage{mathrsfs}
\usepackage[margin=1in]{geometry}
\usepackage{subcaption}
\usepackage{hyperref}
\hypersetup{
    colorlinks=true,
    linkcolor=blue,
    filecolor=blue,      
    urlcolor=blue,
    citecolor=blue,
    pdftitle={Overleaf Example},
  }
  
\usepackage{amsmath}
\usepackage{amsfonts}
\usepackage{amssymb}
\usepackage{graphicx}
\usepackage{tikz-cd}
\usepackage{amsthm}
\usepackage[all]{xy}
\usepackage{adjustbox}
\usepackage{mathtools}
\usepackage[title]{appendix}

\newtheorem{counter}{counter}[section]
\newtheorem{theorem}[counter]{Theorem}

\newtheorem{corollary}[counter]{Corollary}
\newtheorem{lemma}[counter]{Lemma}
\newtheorem{proposition}[counter]{Proposition}

\theoremstyle{definition}
\newtheorem{definition}[counter]{Definition}
\newtheorem{phenomenon}[counter]{Phenomenon}
\newtheorem{notation}[counter]{Notation}

\newtheorem{condition}{Condition}[section]

\title{Nonexistence of exceptional bundles on $\mathbb{P}^{3}$ with maximal possible ranks}
\author{Yeqin Liu}

\begin{document}
\maketitle

\begin{abstract}
We prove that on $\mathbb{P}^{3}$ there is no exceptional bundle with rank $r=2d^{2}+1$ and degree $d$ for every $|d|\geq 4$. In particular, we find a new obstruction for the existence of exceptional bundles other than $r|(2d^{2}+1)$. We also show that there is no exceptional bundle with rank $27$ and degree $11$ to exhibit another different obstruction. 
\end{abstract}

\section{Introduction}
\subsection{Main Theorem}
A vector bundle $E$ on a complex projective variety is called \emph{exceptional} if $\mathrm{Ext}^{*}(E,E)=\mathbb{C}$. Exceptional vector bundles are very useful in the study of vector bundles and derived categories. For example, they form imporant components of semi-orthogonal decompositions of derived categories \cite{BO95}. They are also closely related to the sharp Bogomolov inequality, the classification of the Chern characters of all stable sheaves on a variety. On $\mathbb{P}^{2}$, the sharp Bogomolov inequality has a precise relation with the exceptional bundles \cite{DLP85}. Exceptional bundles can also form generalized Beilinson spectral sequences. In \cite{CHW17}, the authors use this to study the geometry of the moduli spaces of stable sheaves. In \cite{Rud90}, various authors studied exceptional vector bundles on surfaces extensively. However, little is known for threefolds. 

Unlike on $\mathbb{P}^{2}$ where every exceptional bundle is \emph{constructive} by mutations \cite{DLP85}, classifying exceptional bundles on $\mathbb{P}^{3}$ is a challenging problem \cite{Nog91, Rud95}. The same expectation on constructivity requires proving non-existence of non-constructive exceptional bundles. A necessary condition for an exceptional bundle on $\mathbb{P}^{3}$ with rank $r$ and degree $d$ is that $r|(2d^{2}+1)$ \cite{Kul90}. When this is satisfied, few results are known on the existence of such exceptional bundles. In this paper, we develop techniques to study this question. We give the first non-existence result: For a fixed degree $d$, exceptional bundles of maximal possible rank $r=2d^{2}+1$ either are constructive or do not exist.

\begin{theorem}[Theorem \ref{main}]\label{maininintro}
 On $\mathbb{P}^{3}$, there is no exceptional bundle with rank $2d^{2}+1$ and degree $d$ for every $|d|\geq 4$. 
\end{theorem}

In Section \ref{section6}, we exhibit another obstruction (Phenomenon \ref{deform}) for the existence of exceptional bundles which is different from that of Theorem \ref{main}. Using it, we show the following non-existence result that is not covered by Theorem \ref{maininintro}. This also suggests the feasibility to study the constructivity question further, using the techniques developed in this paper.

\begin{theorem}[Theorem \ref{rank27}]\label{rank27inintro}
 On $\mathbb{P}^{3}$, there is no exceptional bundle with rank $27$ and degree $11$. 
\end{theorem}

\subsection{Outline of the paper}

The strategy to prove Theorem \ref{maininintro} is as follows. Suppose $E$ is an exceptional bundle as in Theorem \ref{maininintro}. We cover $\mathbb{P}^{3}$ by the total family of an appropriate pencil of quartic surfaces (Section \ref{section3}). On each quartic surface $X$ in this pencil, we construct a bundle $E_{X}$ that must agree with $E|_{X}$ (Section \ref{section4}). Then we show that all possible bundles on the total family that are swept out by these $E_{X}$ cannot descend to $\mathbb{P}^{3}$ (Section \ref{section5}).

In Section \ref{section6}, we prove Theorem \ref{rank27inintro} by showing that the restrictions of $E_{X}$ and $E_{X'}$ to the base curve of the pencil are not isomorphic (Phenomenon \ref{deform}).

\subsection{Acknowledgements}
I want to thank Izzet Coskun, Benjamin Gould and Sixuan Lou for many valuable discussions and conversations. I want to thank Jack Huizenga and Howard Nuer for many useful comments on an early draft of this paper.

\section{Preliminaries}

In this section we collect facts that relate vector bundles on $\mathbb{P}^{3}$ and quartic surfaces. For more details about exceptional vector bundles, we refer the readers to \cite{Rud90}. For background about $K3$ surfaces and sheaves on them, some good references are \cite{HL97, Huy16}.

\begin{definition}
  A $K3$ surface is a smooth complex projective surface with $ \mathrm{H}^{1}(X, \mathcal{O}_{X})=0$ and $\omega_{X}\cong \mathcal{O}_{X}$.
\end{definition}

A smooth quartic surface in $\mathbb{P}^{3}$ is a $K3$ surface \cite{Huy16}. 

\begin{definition}
Let $X$ be a surface. A vector bundle $E$ on $X$ is \emph{rigid} if $\mathrm{Ext}^{1}_{X}(E, E)=0$. A vector bundle $F$ is \emph{simple} if $\mathrm{Hom}_{X}(F ,F)=\mathbb{C}$. If $X$ is a $K3$ surface, a vector bundle $G$ is \emph{spherical} if it is slope stable (see \cite{HL97}) and rigid.
\end{definition}

In particular, spherical vector bundles on $K3$ surfaces are simple and rigid. The following lemmas relate exceptional vector bundles on $\mathbb{P}^{3}$ and simple rigid vector bundles on quartic surfaces. 

\begin{lemma}[\cite{Kul90}]\label{simplerigid}
Let $E$ be an exceptional vector bundle on $\mathbb{P}^{3}$ and $X\subset \mathbb{P}^{3}$ any quartic surface. Then $E|_{X}$ is a simple rigid vector bundle.
\end{lemma}

\begin{lemma}[\cite{Zub90}]\label{stablespherical}
Let $E$ be an exceptional vector bundle on $\mathbb{P}^{3}$ and $X\subset \mathbb{P}^{3}$ a smooth quartic surface with Picard rank 1. Then $E|_{X}$ is a spherical vector bundle.
\end{lemma}

The following lemma is the crucial ingredient for our proof.

\begin{lemma}\label{sphericalunique}
Let $X$ be a $K3$ surface. Then there is at most one spherical vector bundle on $X$ with a given Chern character.
\end{lemma}

\begin{proof}
  Let $E, F$ be spherical vector bundles on $X$ with the same Chern character. By definition, $\mathrm{ext}^{1}(E,E)=0$ and $\mathrm{ext}^{2}(E, E)=\mathrm{hom}(E, E)=1$, we have $\chi(E,E)=\sum_{i\in \mathbb{Z}}\mathrm{ext}^{i}(E,E)=2$. Since $E, F$ have the same Chern character, by Riemann-Roch, we have $\chi(E, F)=\chi(E, E)=2$. Hence either $\mathrm{hom}(E, F)$ or $\mathrm{ext}^{2}(E, F)=\mathrm{hom}(F, E)$ is nonzero, say $\mathrm{hom}(E, F)\neq 0$. Then there is a nonzero map $E \longrightarrow F$. By stability of $E$ and $F$, we have $E\cong F$. 
\end{proof}

The following fact is known as Mukai's Lemma.

\begin{lemma}[Lemma 5.2, \cite{Bri08}]\label{Mukai}
  Let $\mathcal{A}$ be an abelian category. Suppose
  $$0 \longrightarrow A \longrightarrow B \longrightarrow C \longrightarrow 0 $$
  is a short exact sequence in $\mathcal{A}$ such that $\mathrm{Hom}_{\mathcal{A}}(A, C)=0$. Then we have
  $$\mathrm{ext}^{1}(B,B)\geq \mathrm{ext}^{1}(A, A) + \mathrm{ext}^{1}(C,C). $$
  In particular, if $B$ is rigid, then $A$ and $C$ are rigid.
\end{lemma}

We also need the following lemma, which is a consequence of the cohomology and base change theorems.

\begin{lemma}\label{locally}
  Let $Y$ be a smooth curve and $f: X \longrightarrow Y$ be a flat projective morphism. Let $E, F$ be locally free sheaves on $X$. If
\begin{enumerate}
\item  For all $y\in Y$, we have $E|_{f^{-1}(y)}\cong F|_{f^{-1}(y)}$,
\item  the dimension of $\mathrm{Ext}^{1}_{f^{-1}(y)}(E|_{f^{-1}(y)}, E_{f^{-1}(y)})$ is constant for all $y\in Y$,
\end{enumerate}
then for any $y\in Y$, there exists an open neighborhood $U_{y}\subset Y$ of $y$, such that $E|_{f^{-1}(U)}\cong F|_{f^{-1}(U)}$.
\end{lemma}

\begin{proof}
  For any $y\in Y$ and open subset $U\subset Y$, write $X_{y}:=f^{-1}(y)$ and $X_{U}:=f^{-1}(U)$. For any sheaf $M$ on $X$, write $M_{y}:=M|_{X_{y}}$ and $M|_{U}:=M|_{X_{U}}$.

  Let $y\in Y$ be any point. Choose an affine neighborhood $U$ of $y$. Then $\mathrm{Hom}_{X_{U}}(E_{U},F_{U})= \mathrm{H}^{0}(X_{U}, \underline{Hom}(E_{U}, F_{U}))$. Consider the short exact sequence
  $$0 \longrightarrow \underline{Hom}(E, F(-X_{y})) \longrightarrow \underline{Hom}(E, F) \longrightarrow \underline{Hom}(E, F|_{y})\cong \underline{Hom}_{X_{y}}(E_{y}, F_{y}) \longrightarrow 0. $$
Applying $ \mathrm{H}^{0}(-)$, we get
  $$\mathrm{Hom}(E_{U}, F_{U}) \overset{f}{\longrightarrow} \mathrm{Hom}_{X_{y}}(E_{y}, F_{y}) \longrightarrow  \mathrm{H}^{1}(\underline{Hom}(E_{U}, F_{U}(-X_{y}))) \overset{g}{\longrightarrow}  \mathrm{H}^{1}(\underline{Hom}(E_{U}, F_{U})).$$
  By the Leray spectral sequence, for any vector bundle $M$ on $X_{U}$, $ \mathrm{H}^{1}(X_{U}, M)$ admits a filtration whose graded pieces are $ \mathrm{H}^{0}(U, R^{1}f_{*}M)$ and $ \mathrm{H}^{1}(U, R^{0}f_{*}M)$. Since $U$ is affine, $ \mathrm{H}^{1}(U, R^{0}f_{*}M)=0$ for any coherent sheaf $M$. Hence the map $g$ can be identified with the multiplication of the equation of $y$:
  $$g:  \mathrm{H}^{0}(U, R^{1}f_{*}\underline{Hom}(E_{U}, F_{U})\otimes \mathcal{O}_{Y}(-y)) \overset{y}{\longrightarrow} \mathrm{H}^{0}(U, R^{1}f_{*}\underline{Hom}(E_{U}, F_{U})).$$
  By assumption (2), the dimension of $\mathrm{Ext}^{1}_{X_{y}}(E_{y}, F_{y})\cong \mathrm{Ext}^{1}_{X_{y}}(E_{y}, E_{y})$ is constant for all $y\in Y$. Since $f$ is projective and $E, F$ are flat over $Y$. By the cohomology and base change theorems,
  $R^{1}f_{*}\underline{Hom}(E_{U}, F_{U})$ is locally free.
  Then $g$ is injective, $f$ is surjective. By assumption (1), $E_{y}\cong F_{y}$. Hence we may choose a lift $h\in \mathrm{Hom}_{X_{U}}(E_{U}, F_{U})$ of $\mathrm{id}\in \mathrm{Hom}_{X_{y}}(E_{y},F_{y})$.

  Let $\mathcal{G}'=\mathrm{det}(\underline{\mathrm{Hom}}(E_{U},F_{U}))$. The degeneracy locus of $h$ is the vanishing locus of $\mathrm{det}(h)\in  \mathrm{H}^{0}(X_{U}, \mathcal{G}')$. By assumption (2), $\mathcal{G}|_{y'}\cong \mathcal{O}_{y'}$ for all $y'\in Y$. Hence $\mathcal{G}'=f^{*}\mathcal{L}$ where $\mathcal{L}=f_{*}\mathcal{G}'$ is a line bundle on $Y$. We have the natural identification $ \mathrm{H}^{0}(X_{U}, \mathcal{G}')\cong  \mathrm{H}^{0}(U, \mathcal{L})$. Let $Z\subset U$ be the vanishing locus of $\mathrm{det}(h)\in  \mathrm{H}^{0}(U, \mathcal{L})$, then $y\notin Z$. Over $U_{s}:=U-Z$, $h: E|_{U_{s}} \longrightarrow F|_{U_{s}}$ is an isomorphism.
\end{proof}

\section{Pencils of quartic surfaces}\label{section3}

In this section we will choose an appropriate pencil of quartic surfaces which sweep out $\mathbb{P}^{3}$ (Proposition \ref{pencil}). Let $E$ be an exceptional bundle on $\mathbb{P}^{3}$. On each member $X$ of this pencil, we would like some information about $E|_{X}$. To do this, we cover $X$ by a pencil of curves, such that each curve is the base locus of another pencil of quartic surfaces that contains a smooth member $X'$ with Picard rank 1. Then the restrictions of $E$ to these base curves can be understood, since $E|_{X'}$ is stable spherical (Lemma \ref{stablespherical}). The requirement for this pencil is Condition \ref{condition1}, and we show such pencil exists in Proposition \ref{pencil}. 

Let $\mathbb{P}^{34}=\mathbb{P}^{34}=\mathbb{P} \mathrm{H}^{0}(\mathcal{O}_{\mathbb{P}^{3}}(4H))$ be the space of quartic surfaces in $\mathbb{P}^{3}$. For a very general line $L\subset \mathbb{P}^{34}$, let $\widetilde{\mathbb{P}^{3}_{L}}=\mathrm{BL}_{C}\mathbb{P}^{3}$, where $C$ is the base locus of $L$. Then $\widetilde{\mathbb{P}^{3}_{L}}$ is the total space of the pencil of quartics parameterized by $L$, it admits morphisms $\pi: \widetilde{\mathbb{P}^{3}_{L}} \longrightarrow L$ and $\nu: \widetilde{\mathbb{P}^{3}_{L}} \longrightarrow \mathbb{P}^{3}$. For any $t\in L$, let $X_{t}=\pi^{-1}(t)$. The requirement for $L$ is as follows.

\begin{condition}\label{condition1}
For any $t\in L$, there exists a pencil $L_{t}\subset \mathbb{P}^{34}$ with $t\notin L_{t}$, such that for any $s\in L_{t}$, the pencil $L_{s,t}=\langle X_{s}, X_{t}\rangle$ contains a $K3$ surface with Picard rank 1.
\end{condition}

To study Condition \ref{condition1}, we need to study the Noether-Lefschetz locus for quartic surfaces in $\mathbb{P}^{3}$. In the space of quartic surfaces $\mathbb{P}^{34}$, there is a countable union of hypersurfaces that parametrize smooth quartic surfaces whose Picard rank $>1$ or singular quartic surfaces. We denote these irreducible hypersurfaces by $NL_{i}, i\in \mathbb{Z}_{>0}$, and let $NL=\bigcup_{i=1}^{\infty}NL_{i}$. For any $t\in L$, let $\mathrm{cone}_{i}(t)$ be the union of lines $L'\subset NL_{i}$ that contains $t$.

Condition \ref{condition1} holds for $t$ if and only if $\mathrm{cone}_{i}(t)$ has codimension $\geq 2$ in $\mathbb{P}^{34}$ for all $i$. To see this, first we assume say $\mathrm{cone}_{1}(t)$ has codimension 1. Then for any pencil $L_{t}$, there exists $s\in L_{t}\cap \mathrm{cone}_{1}(t)$. Then the pencil $L_{s,t}\subset \mathrm{cone}_{1}(t)\subset NL_{1}$ has no member whose Picard rank is 1. Conversely, if $\mathrm{cone}_{i}(t)$ has codimension $\geq 2$ for all $i$, then we may choose $L_{t}$ so that $L_{t}\cap \mathrm{cone}_{i}(t)=\emptyset$ for all $i$. For any $s\in L_{t}$, if the pencil $L_{s,t}$ is contained in $NL$, since $L_{s,t}$ is irreducible and $NL$ is a countable union of $NL_{i}$, there exists some $NL_{i}$, say $NL_{1}$, that contains $L_{s,t}$. By definition $L_{s,t}\subset \mathrm{cone}_{1}(t)$, a contradiction. Hence the pencil $L_{s,t}$ is not contained in $NL$, and Condition \ref{condition1} is satisfied.

Note that if $\mathrm{cone}_{i}(t)$ has codimension 1, then it is a hypersurface, hence $\mathrm{cone}_{i}(t)=NL_{i}$ and $[X_{t}]\in NL_{i}$ is a cone point, in the following sense.

\begin{definition}
  For any irreducible subvariety $Y\subset \mathbb{P}^{n}$, a point $p\in Y$ is called a \emph{cone point}, is for any other point $q\in Y$, the line $\overline{pq}\subset Y$. Denote the set of cone points of $Y$ by $CP(Y)$. 
\end{definition}

Then we see that $X_{t}$ satisfies Condition \ref{condition1} if and only if $t\in L\subset \mathbb{P}^{34}$ is not a cone point of $NL_{i}$ for any $i$. Let $CP_{i}=CP(NL_{i})$ and $CP=\bigcup_{i=1}^{\infty}CP_{i}$. We may replace the Condition \ref{condition1} for $L$ with the following equivalent condition.

\begin{condition}\label{condition2}
$$ L\cap CP=\emptyset. $$
\end{condition}

The next lemma shows that cone points are few.

\begin{lemma}\label{few}
Let $Y\subset \mathbb{P}^{n}$ be an irreducible hypersurface whose degree $d\geq 2$. Then $\emph{dim}CP(Y)<\emph{dim}Y$. 
\end{lemma}

\begin{proof}
Let $L\subset \mathbb{P}^{n}$ be a general line. Then $L$ intersects $Y$ transversely in $d$ distinct points $p_{1}, \cdots, p_{d}$. If $d\geq 2$, then $p_{i}$ are not cone points. 
\end{proof}

In the following we shall prove that $NL_{i}$ is not a hyperplane. 

\begin{lemma}\label{nothp}
  For every $i$, $NL_{i}$is not a hyperplane.
\end{lemma}

\begin{proof}
Let $v: (\mathbb{P}^{3})^{*} \hookrightarrow \mathbb{P}^{34}$ be the Veronese embedding sending a linear form $f$ to $f^{4}$.  Note that the closure $\overline{NL_{i}}$ of $NL_{i}$ is a hypersurface in $\mathbb{P}^{34}$. Then $\overline{NL_{i}}\cap v((\mathbb{P}^{3})^{*})$ contains at least one point $p$. The orbit of $p$ under $\mathbb{P}\mathrm{GL}_{4}$ is $v((\mathbb{P}^{3})^{*})$. Since $\mathbb{P}\mathrm{GL}_{4}$ is connected and Picard rank of quartic surfaces are preserved under $\mathbb{P}\mathrm{GL}_{4}$ action, the orbit of $NL_{i}$ must be $NL_{i}$. Hence $\overline{NL_{i}}$ contains $v((\mathbb{P}^{3})^{*})$. Since the Veronese embedding is non-degenerate, $\overline{NL_{i}}$ cannot be a hyperplane.
\end{proof}
By Lemma \ref{few} and Lemma \ref{nothp}, $CP_{i}$ has codimension $\geq 2$ in $\mathbb{P}^{34}$ for every $i$. Since $CP$ is a contable union of $CP_{i}$, there exists a line $L\subset \mathbb{P}^{34}$ that satisfies Condition \ref{condition2}. Furthermore, such $L$ is very general. Therefore we summarize the work of this section in the following proposition.

\begin{proposition}\label{pencil}
A very general pencil $L\subset \mathbb{P}^{34}$ satisfies Condition \ref{condition1}.
\end{proposition}

\section{Description on quartic surfaces}\label{section4}

Let $E$ be an exceptional vector bundle on $\mathbb{P}^{3}$ of rank $2d^{2}+1$ and degree $d$. By Proposition \ref{pencil}, we may choose a pencil $L\subset \mathbb{P}^{34}$ that satisfies Condition \ref{condition1}. As outlined in the introduction, we will construct a bundle $E_{X}$ on each member $X$ of this pencil, such that the following proposition holds. 

\begin{proposition}[Proposition \ref{EXdunique}]\label{EXduniquesimple}
On each $X\in L$, there exists a bundle $E_{X}$ such that $E|_{X}\cong E_{X}$. 
\end{proposition}

Proposition \ref{EXduniquesimple} is the core of this paper, we outline the proof here. We will explictly construct $E_{X}$ first. For any $t\in L$ we cover $X_{t}$ by a pencil $L'$ of curves $\{C_{s}\}_{s\in L'}$. By the choice of the pencil $L$ (Condition \ref{condition1}), $E|_{X}$ and $E_{X}$ can be made isomorphic fiberwise in the family $\{C_{s}\}_{s\in L'}$ (Lemma \ref{split}). We will extend this isomorphism to an open neighborhood (Lemma \ref{locally}), then the bundles on the total family are glued from the local families. Then we will show this glueing is unique by carefully analyzing the relative Harder-Narasimhan filtration.

In this section we fix any $t\in L$, let $X=X_{t}$, $E=E|_{X}$, $L'=L_{t}$ and $L'_{s}=L_{s,t}$ as in Condition \ref{condition1}. Consider the pencil of curves $\{C_{s}=X\cap X_{s}\}_{s\in L'}$ on $X$. The base locus of this pencil is 64 points $q_{1}, \cdots, q_{64}$. Let $\widetilde{X}=\mathrm{BL}_{q_{1}, \cdots, q_{64}}X$ be the total space of this pencil of curves and $\Sigma$ be the exceptional divisor. Then $\widetilde{X}$ admits morphisms $v: \widetilde{X} \longrightarrow X$ and $p: \widetilde{X} \longrightarrow L'$.

\subsection{Modifications}
We define the following notion and discuss its properties when applied to $\widetilde{X}$ in this subsection.

\begin{definition}\label{modification}
  Let $f: Z \rightarrow S$ be a flat projective morphism and $M$ be a vector bundle on $Z$. A vector bundle $M'$ is called a \emph{modification} of $M$, if for all $s\in S$, there exists an open neighborhood $U_{s}$, such that $M'|_{f^{-1}(U_{s})}\cong M|_{f^{-1}(U_{s})}$. A modification $M'$ is called \emph{trivial} if $M'=M\otimes f^{*}\mathcal{L}$ for some line bundle $\mathcal{L}$ on $S$. 
\end{definition}

\begin{notation}\label{endomorphismsheaf}
  Under the settings in Definition \ref{modification}, we let $\mathcal{A}_{M}:=f_{*}\underline{End}(M)$ and $\mathcal{G}_{M}\subset \mathcal{A}_{M}$ be the sheaf of groups whose local sections are multiplicative invertible sections of $\mathcal{A}_{M}$.
\end{notation}
Suppose $\mathrm{End}_{f^{-1}(s)}(M|_{f^{-1}(s)})$ in Definition \ref{modification} has constant dimension for all $s\in S$. Then modifications of $M$ are parametrized by the (nonabelian) cohomology set $ \mathrm{H}^{1}(S, \mathcal{G}_{M})$. We first note the following two lemmas about modifications.

\begin{lemma}\label{fiberwisesimple}
  Using the notations in Definition \ref{modification}, if $M$ is fiberwise simple, i.e.
  $$\mathrm{End}_{f^{-1}(s)}(M_{f^{-1}(s)})=\mathbb{C}$$
  for all $s\in S$, then all modifications of $M$ are trivial.
\end{lemma}

\begin{proof}
Note that $\mathcal{A}_{M}=\mathcal{O}_{S}$ and $\mathcal{G}_{M}=\mathcal{O}_{S}^{*}$, we have $ \mathrm{H}^{1}(S, \mathcal{G}_{M})= \mathrm{H}^{1}(S, \mathcal{O}_{S}^{*})=\mathrm{Pic}(S)$. 
\end{proof}

\begin{lemma}\label{fiberwisetrivial}
  Let $Q$ be a rigid and fiberwise simple vector bundle on $\widetilde{X}$ and $M=Q^{\oplus p}$ for some $p>0$. Then all rigid modifications of $M$ with respect to $p: \widetilde{X} \longrightarrow L'$ are trivial. In particular, for every line bundle $\mathcal{L'}$ on $L'$, all rigid modifications of $p^{*}\mathcal{L'}^{\oplus p}$ are trivial.
\end{lemma}

\begin{proof}
  Recall that $\widetilde{X}=\mathrm{BL}_{q_{1}, \cdots, q_{64}}X$ and $\Sigma$ is the exceptional divisor of $v: \widetilde{X} \rightarrow X$. Let $F$ be the fiber class of $p: \widetilde{X} \rightarrow L'$. 

  Since $Q$ is fiberwise simple, $\mathcal{A}_{M}=\mathcal{O}_{L'}\otimes \mathrm{Mat}_{p}$ and $\mathcal{G}_{M}=\mathrm{GL}_{p}$, we have
  $$ \mathrm{H}^{1}(L', \mathcal{G}_{M})= \mathrm{H}^{1}(L', \mathrm{GL}_{p})=\mathrm{Bun}_{p}(L'), $$
  where $\mathrm{Bun}_{p}(L')$ is the set of rank $p$ vector bundles on $L'$.
  Hence any modification $M'$ of $M$ is of the form $\oplus_{i=1}^{p}Q(a_{i}F)$. Now assume some $a_{i}\neq a_{j}$, say $d:=a_{1}-a_{2}\geq 1$, we need to show $M'$ is not rigid, i.e. $\mathrm{Ext}^{1}_{\widetilde{X}}(M',M')\neq 0$. It suffices to show $\mathrm{Ext}^{1}_{\widetilde{X}}(p^{*}\mathcal{O}_{L'}(a_{2}), p^{*}\mathcal{O}_{L'}(a_{1}))\neq 0$, since it is a summand of $\mathrm{Ext}^{1}_{\widetilde{X}}(M', M')$. By the Leray spectral sequence, it suffices to show
  $$\mathrm{H}^{0}(R^{1}p_{*}\underline{Hom}(Q(a_{2}), Q(a_{1})))= \mathrm{H}^{0}(R^{1}p_{*}\underline{End}(Q)\otimes \mathcal{O}_{L'}(d))\neq 0. $$
  Since $\mathcal{O}_{\widetilde{X}}$ is a summand of $\underline{End}(Q)$, it suffices to show
  $$ \mathrm{H}^{0}(R^{1}p_{*}\mathcal{O}_{\widetilde{X}}\otimes \mathcal{O}_{L'}(d))\neq 0.$$

   The relative canonical divisor of $p: \widetilde{X} \longrightarrow L'$ is $K_{rel}=K_{\widetilde{X}}-p^{*}K_{L'}=\Sigma+2F=4H+F$. By relative Serre duality, we have
  $$ R^{1}p_{*}\mathcal{O}_{\widetilde{X}}=p_{*}\mathcal{O}_{\widetilde{X}}(4H)^{*}(-1)=\mathcal{O}_{L'}(-1)^{\oplus 32}\oplus \mathcal{O}_{L'}(-2). $$
  Hence
  $$ \mathrm{H}^{0}(R^{1}p_{*}\mathcal{O}_{\widetilde{X}}\otimes \mathcal{O}_{L'}(d))= \mathrm{H}^{0}(\mathcal{O}_{L'}(d-1)^{\oplus 32}\oplus \mathcal{O}_{L'}(d-2)).$$
  By assumption, $d\geq 1$. Hence $d-1\geq 0$ and $ \mathrm{H}^{0}(R^{1}p_{*}\mathcal{O}_{\widetilde{X}}\otimes \mathcal{O}_{L'}(d))\neq 0$. 
\end{proof}

The next lemma shows rigid modifications can be checked fiberwise. 

\begin{lemma}
Let $E, F$ be simple rigid vector bundles on $X$. If $E|_{C_{s}}\cong F|_{C_{s}}$ for all $s\in L'$, then $F$ is a modification of $E$.
\end{lemma}

\begin{proof}
 Since $E=v_{*}v^{*}(E), F=v_{*}v^{*}(F)$, the lemma will follow from Lemma \ref{locally}. By assumption, condition (1) in Lemma \ref{locally} is satisfied, it suffices to check condition (2). By Lemma \ref{simplerigid}, $E$ is a simple rigid vector bundle on $X$. For any $s\in L'$, applying $\mathrm{Hom}_{X}(E,-)$ to the short exact sequence
  $$0 \longrightarrow E(-4H) \longrightarrow E \longrightarrow E|_{C_{s}} \longrightarrow 0 ,$$
  we have the long exact sequence
  
  \begin{multline*}
   0 \longrightarrow \mathrm{Hom}_{X}(E,E(-4H)) \longrightarrow \mathrm{Hom}_{X}(E,E) \longrightarrow \mathrm{Hom}_{C_{s}}(E_{s}, E_{s}) \\
    \longrightarrow \mathrm{Ext}^{1}_{X}(E, E(-4H)) \longrightarrow \mathrm{Ext}^{1}_{X}(E,E)=0 .
    \end{multline*}    
  Hence 
  $$\mathrm{hom}_{C_{s}}(E_{s},E_{s})=\mathrm{hom}_{X}(E, E)+\mathrm{ext}^{1}_{X}(E,E(-4H))-\mathrm{hom}_{X}(E, E(-4H))$$
  is independent of $s\in L'$. By Riemann-Roch, $\mathrm{ext}^{1}(E_{s},E_{s})$ is also independent of $s\in L'$.
\end{proof}

We will need the following version of Theorem 1.2 in \cite{Pol11}, which we prove from a different perspective. The relation of these two versions will be explained in Corollary \ref{Polcor}.

\begin{proposition}\label{Pol}
Let $E,F$ be simple vector bundles on $X$ such that $E|_{C_{s}}\cong F|_{C_{s}}$ for all $s\in L'$. If $\mathrm{Ext}^{1}_{X}(E, E(-4H))=0$, then $E \cong F$.

\end{proposition}

\begin{proof}
Using Notation \ref{endomorphismsheaf}, the set of modifications of $v^{*}(E)$ is $ \mathrm{H}^{1}(L, \mathcal{G}_{v^{*}(E)})$. By Lemma \ref{locally}, $v^{*}(F)\in  \mathrm{H}^{1}(L, \mathcal{G}_{v^{*}(E)})$. Applying $\mathrm{Hom}_{X}(E,-)$ to the following short exact sequence
$$0 \longrightarrow E(-4H) \longrightarrow E \longrightarrow E|_{C_{s}} \longrightarrow 0, \forall s\in L',$$
we get 
$$\mathbb{C}=\mathrm{Hom}_{X}(E, E) \longrightarrow \mathrm{Hom}_{X}(E, E_{C_{s}})\cong \mathrm{Hom}_{C_{s}}(E|_{C_{s}}, E|_{C_{s}}) \longrightarrow \mathrm{Ext}^{1}_{X}(E, E(-4H))=0.$$
Hence $\mathrm{Hom}_{C_{s}}(E|_{C_{s}}, E|_{C_{s}})=\mathbb{C}$ for all $s\in L'$, $\mathcal{A}_{v^{*}(E)}=\mathcal{O}_{L'}$, and $\mathcal{G}_{v^{*}(E)}\cong \mathcal{O}_{L'}^{*}$. Therefore
  $$\mathrm{H}^{1}(L', \mathcal{O}_{L'}^{*})=\mathrm{Pic}(L')=\mathbb{Z}$$
  is in one-to-one correspondence with line bundles on $L'$. Let $\Sigma$ be the exceptional divisor of $v$. Note that the restrictions to $\Sigma$ of different twists of $v^{*}(E|_{X})$ by line bundles on $L'$ are non-isomorphic. Since $v^{*}(E)|_{\Sigma}$ and $v^{*}(F)|_{\Sigma}$ are both trivial, we have $v^{*}(E|_{X})\cong v^{*}(F|_{X})$.
\end{proof}

\subsection{The vector bundle $E_{X}$}
In this subsection we define the bundle $E_{X}$ on $X$ and discuss its properties. Then we will prove the main result of this section (Proposition \ref{EXdunique}).

In the following we assume $d\geq 0$, the case $d< 0$ can be dealt by taking the dual. Let $v_{d}$ be the Mukai vector $(2d^{2}+1, -dH, 1) \in H^{*}_{alg}(X)$.
We define $E_{X, d}$ on $X$ as the sheaf that fits into the following exact sequence
\begin{equation}\label{EXd}
  0 \longrightarrow E_{X, d} \longrightarrow \mathcal{O}_{X}\otimes \mathrm{H}^{0}(\mathcal{O}_{X}(dH)) \overset{ev}{\longrightarrow} \mathcal{O}_{X}(dH) \longrightarrow 0,
\end{equation}
and let $\mathcal{E}_{\widetilde{X}, d}=v^{*}E_{X, d}$ on $\widetilde{X}$.  Let $C\in |4H|_{X}$ be any curve. We also define $F_{C, d}$ as the sheaf on $C$ that fits into the following exact sequence
\begin{equation}\label{FCd}
  0 \longrightarrow F_{C, d} \longrightarrow \mathcal{O}_{C}\otimes  \mathrm{H}^{0}(C, \mathcal{O}_{C}(dH)) \longrightarrow \mathcal{O}_{C}(dH) \longrightarrow 0,
\end{equation}
and let $\mathcal{F}_{\widetilde{X}, d}$ on $\widetilde{X}$ be the sheaf swept out by $F_{C, d}$. In other words, $\mathcal{F}_{\widetilde{X}, d}$ fits into the following exact sequence
\begin{equation}\label{FXd}
  0 \longrightarrow \mathcal{F}_{\widetilde{X}, d} \longrightarrow p^{*}(p_{*}\mathcal{O}_{\widetilde{X}}(dH)) \overset{ev}{\longrightarrow} \mathcal{O}_{\widetilde{X}}(dH) \longrightarrow 0.
\end{equation}
The first observation is that $E_{X, d}$ is simple and rigid. This is a consequence of a more general proposition.

\begin{proposition}\label{derivedmutation}
  Let $Y$ be a $K3$ surface and $A, B\in \mathcal{D}^{b}(Y)$ be simple rigid objects. Then the mapping cone of $ev: \mathrm{RHom}(A, B)\otimes A \longrightarrow B$ is simple and rigid.
\end{proposition}

\begin{proof}
  Let $C$ be the mapping cone. Let $S=\mathbb{C}\oplus \mathbb{C}[-2]\in \mathcal{D}^{b}(Y)$ and $p_{i}: S \rightarrow \mathcal{H}^{i}(S)[-i]$ be the projection map. We need to show $\mathrm{RHom}(C, C)\cong S$. Applying $\mathrm{RHom}(-, -)$ to the following exact triangle in both arguments:
  $$C \longrightarrow \mathrm{RHom}(A, B)\otimes A \overset{ev}{\longrightarrow} B. $$
  Note that $\mathrm{RHom}(A, A)\cong \mathrm{RHom}(B, B) \cong S$ by assumption,
  we get a commutative diagram whose rows and columns are all exact triangles:
  \\
  
  \begin{adjustbox}{scale=0.85, center}
\begin{tikzcd}
  \mathrm{RHom}(C, C) \arrow[r] & \mathrm{RHom}(C, A)\otimes \mathrm{RHom}(A, B)\arrow[r] & \mathrm{RHom}(C, B) \\
  \mathrm{RHom}(A, B)^{*}\otimes \mathrm{RHom}(A, C) \arrow[r]\arrow[u] & \mathrm{RHom}(A, B)^{*}\otimes S \otimes \mathrm{RHom}(A, B) \arrow[r,  "p_{0}"]\arrow[u] & \mathrm{RHom}(A, B)^{*}\otimes \mathrm{RHom}(A, B) \arrow[u] \\
  \mathrm{RHom}(B, C) \arrow[r]\arrow[u] & \mathrm{RHom}(B, A)\otimes \mathrm{RHom}(A, B) \arrow[r]\arrow[u, "p_{0}^{*}"] & \mathrm{RHom}(B, B)\cong S \arrow[u].  
\end{tikzcd}
\end{adjustbox}
\\
\\
For any $M, N\in \mathcal{D}^{b}(Y)$, by Serre Duality we have $\mathrm{RHom}(M, N)^{*}\cong \mathrm{RHom}(N, M)[2]$. Hence applying dual will flip the diagram about the anti-diagonal and shift its degree by 2. Therefore $\mathrm{RHom}(C, C)\cong \mathrm{RHom}(B, B)\cong S$.

\end{proof}

\begin{corollary}\label{Edrigid}
The sheaf $E_{X, d}$ is simple and rigid for $d> 0$. In particular, it is locally free.
\end{corollary}

\begin{proof}
  By the Kodaira Vanishing Theorem, we have $ \mathrm{H}^{1}(\mathcal{O}_{X}(dH))=0$. Hence we have $\mathrm{RHom}(\mathcal{O}_{X}, \mathcal{O}_{X}(dH))= \mathrm{H}^{0}(\mathcal{O}_{X}(dH))$ for $d> 0$. The claim follows from Proposition \ref{derivedmutation}. 
\end{proof}

Consider the following exact sequence
$$0 \longrightarrow \mathcal{O}_{X}((d-4)H) \longrightarrow \mathcal{O}_{X}(dH) \longrightarrow \mathcal{O}_{C}(dH) \longrightarrow 0. $$
Taking cohomology, we have
$$0 \longrightarrow  \mathrm{H}^{0}(\mathcal{O}_{X}((d-4)H)) \longrightarrow  \mathrm{H}^{0}(\mathcal{O}_{X}(dH)) \longrightarrow  \mathrm{H}^{0}(C, \mathcal{O}_{C}(dH)) \longrightarrow  \mathrm{H}^{1}(\mathcal{O}_{X}((d-4)H))=0. $$
Hence $\mathrm{h}^{0}(C, \mathcal{O}_{C}(dH))$ is independent of the choice of $C\in |4H|_{X}$. Furthermore, when $d\geq 4$, we see that $ \mathrm{H}^{0}(C, \mathcal{O}_{C}(dH))$ is a strict quotient of $ \mathrm{H}^{0}(\mathcal{O}_{X}(dH))$. Hence $F_{C, d}$ is a strict quotient of $E_{X, d}|_{C}=\mathcal{E}_{d}|_{C}$. We have the following commutative diagram
\[\begin{tikzcd}
    & \mathrm{ker}\arrow[d] & K_{\widetilde{X}, d}\arrow[d] & 0 \arrow[d]& \\
    0\arrow[r] & \mathcal{E}_{\widetilde{X}, d} \arrow[r]\arrow[d] & \mathcal{O}_{\widetilde{X}}\otimes  \mathrm{H}^{0}(X, \mathcal{O}_{X}(dH)) \arrow[r]\arrow[d] & \mathcal{O}_{\widetilde{X}}(dH) \arrow[r]\arrow[d, "="] & 0 \\
    0\arrow[r] & \mathcal{F}_{\widetilde{X}, d} \arrow[r]\arrow[d] & p^{*}(p_{*}\mathcal{O}_{\widetilde{X}}(dH)) \arrow[r]\arrow[d] & \mathcal{O}_{\widetilde{X}}(dH) \arrow[r]\arrow[d] & 0 \\
    & 0 & 0 & 0 &
  \end{tikzcd}\]
On each $C\in |4H|_{X}$, $K_{\widetilde{X}, d}|_{C}\cong \mathcal{O}_{C}\otimes  \mathrm{H}^{0}(\mathcal{O}_{X}(-C))$. However, we will see that $K_{\widetilde{X}, d}$ is not a constant family. Consider the multiplication map
$$ \mathrm{H}^{0}(\mathcal{O}_{X}((d-4)H))\otimes  \mathrm{H}^{0}(\mathcal{O}_{X}(4H)) \longrightarrow  \mathrm{H}^{0}(\mathcal{O}_{X}(dH)). $$
This induces a morphism
$$|4H|_{X} \longrightarrow G(\mathrm{h}^{0}((d-4)H), \mathrm{H}^{0}(\mathcal{O}_{X}(dH))).$$
Precomposing this map with the inclusion $L'\subset |4H|_{X}$, we have a morphism
$$m: L' \longrightarrow  G(\mathrm{h}^{0}((d-4)H),  \mathrm{H}^{0}(\mathcal{O}_{X}(dH))). $$
Let $S$ be the tautological subbundle of $ G(\mathrm{h}^{0}((d-4)H),  \mathrm{H}^{0}(\mathcal{O}_{X}(dH)))$, then by construction $K_{\widetilde{X}, d}= m^{*}S$. Explicitly, $K_{\widetilde{X}, d}\cong  \mathrm{H}^{0}(\mathcal{O}_{X}((d-4)H))\otimes \mathcal{O}_{\widetilde{X}}(-F)$, where $F$ is the fiber class of $p$. Hence we have the following exact sequence
\begin{equation}\label{relHN}
  0 \longrightarrow  \mathrm{H}^{0}(\mathcal{O}_{\widetilde{X}}((d-4)H))\otimes \mathcal{O}_{\widetilde{X}}(-F) \longrightarrow \mathcal{E}_{\widetilde{X}, d} \longrightarrow \mathcal{F}_{\widetilde{X}, d} \longrightarrow 0.
\end{equation}
We will also need the following observations.

\begin{lemma}\label{split}
  The sequence (\ref{relHN}) splits on each fiber. More generally, for every $C\in |4H|_{X}$ we have
  $$E_{X, d}|_{C}\cong ( \mathrm{H}^{0}(\mathcal{O}_{X}((d-4)H))\otimes \mathcal{O}_{C})\oplus F_{C}.$$
\end{lemma}

\begin{proof}
  On each $C\in |4H|_{X}$, the evaluation map $ev: \mathrm{H}^{0}(\mathcal{O}_{X}(dH)) \longrightarrow \mathcal{O}_{X}(dH)$ vanishes on the subbundle
  $$ \mathrm{H}^{0}(\mathcal{O}_{X}((d-4)H))\otimes \mathcal{O}_{X} \cong  \mathrm{H}^{0}(\mathcal{O}_{X}(dH-C))\otimes \mathcal{O}_{X}\subset  \mathrm{H}^{0}(\mathcal{O}_{X}(dH))\otimes \mathcal{O}_{X}.$$
  Hence the sequence (\ref{relHN}) splits on $C$.
\end{proof}

\begin{lemma}\label{FXdstable}
  The sequence (\ref{relHN}) is the relative Harder-Narasimhan filtration of $\mathcal{E}_{\widetilde{X}, d}$, in the sense that the restriction of (\ref{relHN}) to every smooth fiber is the Harder-Narasimhan filtration of the restriction of $\mathcal{E}_{\widetilde{X}, d}$.
\end{lemma}

\begin{proof}
 Note that for $d\geq 4$, $F_{C_{s}, d}$ is stable on every smooth fiber $C_{s}$ by \cite{But94}. Since $ \mathrm{H}^{0}(\mathcal{O}_{X}((d-4)H))\otimes \mathcal{O}_{C_{s}}$ is semistable and $\mu(F_{C_{s}, d})<\mu(\mathcal{O}_{C_{s}})$, the lemma follows from the uniqueness of the Harder-Narasimhan filtration. 
\end{proof}

Now we show the main result of this section.

\begin{proposition}\label{EXdunique}
  Suppose $E$ is an exceptional bundle on $\mathbb{P}^{3}$ with rank $2d^{2}+1$ and degree $-d$ for some $d>0$. Then for any smooth quartic $X$ in a pencil that satisfies Condition \ref{condition1}, we have $E|_{X}\cong E_{X, d}$ (defined in (\ref{EXd})). 
\end{proposition}

\begin{proof}
  
  The proof will be separated into several steps. First we will see that $v^{*}(E|_{X})$ is a rigid modification of $v^{*}(E_{X, d})$ (Definition \ref{modification}). Then we find an invariant for such modifications. Next, we show that the invariants of $v^{*}(E|_{X})$ and $v^{*}(E_{X, d})$ are equal. Finally, this implies that $v^{*}(E|_{X})\cong v^{*}(E_{X, d})$. Since when $d\leq 3$ the relative Harder-Narasimhan filtration of $v^{*}(E|_{X})$ is different from the case $d\geq 4$, we deal with this case first.

   Let $X$ be a quartic surface in a pencil that satisfies Condition \ref{condition1}. The Mukai vector of $E|_{X}$ is of the form $v_{d}=(2d^{2}+1, -dH, a)$. Since $E|_{X}$ is simple and rigid, $v_{d}^{2}=-2$, hence $a=1$.
  \\
\\
  \textbf{Step 1}: Prove the proposition for $0< d\leq 3$.
  
  In this case, there exists a constructive exceptional bundle $E$ on $\mathbb{P}^{3}$ that fits into the exact sequence
  $$0 \longrightarrow E \longrightarrow  \mathrm{H}^{0}(\mathcal{O}_{\mathbb{P}^{3}}(dH))\otimes \mathcal{O}_{\mathbb{P}^{3}} \longrightarrow \mathcal{O}_{\mathbb{P}^{3}}(dH) \longrightarrow 0. $$
  Since $d\leq 3$, the restriction map $ \mathrm{H}^{0}(\mathcal{O}_{\mathbb{P}^{3}}(dH)) \rightarrow  \mathrm{H}^{0}(\mathcal{O}_{X}(dH))$ is an isomorphism, and $E|_{X}$ fits into the following exact sequence
  $$0 \longrightarrow E|_{X} \longrightarrow  \mathrm{H}^{0}(\mathcal{O}_{X}(dH))\otimes \mathcal{O}_{X} \longrightarrow \mathcal{O}_{X}(dH) \longrightarrow 0. $$
  Since $E$ is constructive, by Theorem 2.3 of \cite{BP93}, $\mathrm{Hom}_{\mathbb{P}^{3}}(E,E(4H))=0$. By Proposition \ref{Pol}, any exceptional bundle $F$ with the same rank and degree as $E$ has $F\cong E$, hence Proposition \ref{EXdunique} is proved for $d\leq 3$. In the following we assume $d\geq 4$. 
  \\
  \\
  \textbf{Step 2}: Show that $v^{*}(E|_{X})$ is a rigid modification of $v^{*}(E_{X, d})$.
  
  If $X$ is smooth with Picard rank 1, then by Lemma \ref{stablespherical}, $E|_{X}\in M_{X, H}(v_{d})$. By \cite{BM14a} and Corollary 4.13 of \cite{Liu22} (or a direct computation), $E_{X, d}$ is stable for any $d>0$. Hence $E|_{X}\cong E_{X, d}$ when $\mathrm{Pic}(X)=\mathbb{Z}H$.
  
  For singular $X$ or $X$ with higher Picard rank, let $L'\subset |4H|_{X}$ be as in the beginning of the section and $\widetilde{X}$ be the universal family of $L'$. Since $L'$ satisfies Condition \ref{condition1}, for any $s\in L'$, there exists $u\in L'_{s}=\langle X, X_{s} \rangle$ such that $X_{u}$ is smooth with Picard rank 1. In particular $X_{u}\neq X$, hence $C_{s}=X\cap X_{u}$. By Lemma \ref{split}, we have
  $$(E|_{X})_{C_{s}}=E|_{C_{s}}=(E|_{X_{u}})|_{C_{s}}\cong E_{X_{u}, d}|_{C_{s}}\cong (\mathrm{H}^{0}(\mathcal{O}_{X}((d-4)H))\otimes \mathcal{O}_{C_{s}}) \oplus F_{C_{s}, d}\cong E_{X, d}|_{C_{s}}. $$
  Hence $v^{*}(E|_{X})$ is a modification of $\mathcal{E}_{X, d}$ (Definition \ref{modification}). By Lemma \ref{simplerigid}, it is rigid.
  \\
  \\
  \textbf{Step 3}: Find the invariant.

  The relative Harder-Narasimhan factors of $v^{*}(E|_{X})$ are modifications of the relative Harder-Narasimhan factors of $\mathcal{E}_{X, d}$. By Lemma \ref{FXdstable}, (\ref{relHN}) is the relative Harder-Narasimhan filtration of $\mathcal{E}_{X, d}$. Hence the relative Harder-Narasimhan filtration of $v^{*}(E|_{X})$ is of the following form
  \begin{equation}\label{GG}
    0 \longrightarrow G_{1} \longrightarrow v^{*}(E|_{X}) \longrightarrow G_{2} \longrightarrow 0,
  \end{equation}
  where $G_{1}$ is a modification of $ \mathrm{H}^{0}(\mathcal{O}_{\widetilde{X}}((d-4)H))\otimes \mathcal{O}_{\widetilde{X}}(-F)$ and $G_{2}$ is a modification of $\mathcal{F}_{\widetilde{X}, d}$. By Lemma \ref{fiberwisesimple}, $G_{2}\cong \mathcal{F}_{\widetilde{X}, d}\otimes \mathcal{O}_{\widetilde{X}}(bF)$ for some $b\in \mathbb{Z}$. Since $G_{1}, G_{2}$ are relative Harder-Narasimhan factors by Lemma \ref{FXdstable}, we have $p_{*}\underline{Hom}(G_{1}, G_{2})=0$, hence
  $$\mathrm{Hom}_{\widetilde{X}}(G_{1}, G_{2})=  \mathrm{H}^{0}(p_{*}\underline{Hom}(G_{1}, G_{2}))=0. $$
  Since $v^{*}(E|_{X})$ is rigid, by Mukai's Lemma (Lemma \ref{Mukai}) $G_{1}$ is also rigid. By Lemma \ref{fiberwisetrivial}, we have $G_{1}\cong  \mathrm{H}^{0}(\mathcal{O}_{\widetilde{X}}((d-4)H))\otimes \mathcal{O}_{\widetilde{X}}(-F+aF)$ for some $a\in \mathbb{Z}$.
  Let $p:=\mathrm{h}^{0}(\mathcal{O}_{X}((d-4)H))$, $c:=b-a$, and
  $$\mathcal{G}:=\underline{Hom}(\mathcal{F}_{\widetilde{X}, d}, \mathcal{O}_{\widetilde{X}}(-F))=\mathcal{F}_{\widetilde{X}, d}^{*} \otimes \mathcal{O}_{\widetilde{X}}(-F). $$
  Then by the Leray spectral sequence, $\mathrm{Ext}^{1}(\mathcal{F}_{\widetilde{X}, d}, \mathcal{O}_{\widetilde{X}}(-F))$ admits a filtrations whose factors are
  $ \mathrm{H}^{0}(R^{1}p_{*}\mathcal{G})$ and $ \mathrm{H}^{1}(p_{*}\mathcal{G})$. The invariant that we will use is $\mathrm{h}^{1}(R^{1}p_{*}\mathcal{G})$. We will show that the invariants of $v^{*}(E|_{X})$ and $v^{*}(E_{X, d})$ are equal. Explicitly, the goal is to show
  \begin{equation}\label{inv}
    \mathrm{h}^{1}(R^{1}p_{*}\mathcal{G})=\mathrm{h}^{1}(R^{1}p_{*}\underline{Hom}(\mathcal{F}_{\widetilde{X}, d}, \mathcal{O}_{\widetilde{X}}(-F))).
  \end{equation}
  \\
  \textbf{Step 4}: Prove (\ref{inv}).

  First note that $\mathrm{ext}^{1}(\mathcal{F}_{\widetilde{X}, d}, \mathcal{O}_{\widetilde{X}}(-F))=p$ for the following reason. Since $\mathcal{E}_{\mathcal{X}, d}$ is rigid, any general classes $\delta_{1}, \delta_{2}\in \mathrm{Ext}^{1}(\mathcal{F}_{\widetilde{X}, d}, \mathcal{O}_{\widetilde{X}}(-F)^{\oplus p})$ have the middle term isomorphic to $\mathcal{E}_{\widetilde{X}, d}$. Since $\mathrm{Hom}(\mathcal{F}_{\widetilde{X}, d}, \mathcal{O}_{\widetilde{X}}(-F))=0$, we have that $\delta_{1}=g(\delta_{2})$ for some $g\in \mathrm{GL}_{p}$ under the natural action of $\mathrm{GL}_{p}$ on $\mathrm{Ext}^{1}(\mathcal{F}_{\widetilde{X}, d}, \mathcal{O}_{\widetilde{X}}(-F)^{\oplus p})$ (say, by Lemma 6.3 in \cite{CH18}). Hence we have
  $$\mathrm{dim}\mathrm{Ext}^{1}(\mathcal{F}_{\widetilde{X}, d}, \mathcal{O}_{\widetilde{X}}(-F)^{\oplus p}) \leq \mathrm{dim}\mathrm{GL}_{p}=p^{2}, $$
  namely $\mathrm{ext}^{1}(\mathcal{F}_{\widetilde{X}, d}, \mathcal{O}_{\widetilde{X}}(-F))\leq p$. However, if $\mathrm{ext}^{1}(\mathcal{F}_{\widetilde{X}, d}, \mathcal{O}_{\widetilde{X}}(-F))<p$, then the sequence (\ref{relHN}) partially splits, i.e. there exists some summand $\mathcal{O}_{\widetilde{X}}(-F)\subset \mathcal{O}_{\widetilde{X}}(-F)^{\oplus p}$, such that the restriction of (\ref{relHN}) to it admits a section. In particular, $\mathcal{O}_{\widetilde{X}}(-F)$ is a summand of $\mathcal{E}_{\widetilde{X}, d}$. Hence $\mathcal{E}_{\widetilde{X}, d}$ is not simple, this contradicts Corollary \ref{Edrigid}. 
  
  Next, we claim that $ \mathrm{H}^{0}(R^{1}p_{*}\mathcal{G})=0$. Suppose not, then there exists a nonzero class $\delta\in \mathrm{H}^{0}(R^{1}p_{*}\mathcal{G}^{\oplus p})$ whose middle term is denoted by $E'$. By construction, there exists some fiber $C_{s}$, such that the restriction of $\delta$ to $C_{s}$ is nonsplit. 
  However, since the sequence (\ref{relHN}) in $\mathrm{Ext}^{1}(\mathcal{F}_{\widetilde{X}, d}, \mathcal{O}_{\widetilde{X}}(-F)^{\oplus p})$, $E'$ is a deformation of $E_{\widetilde{X}, d}$. Since $E_{\widetilde{X}, d}$ is rigid, $E' \cong E_{\widetilde{X}, d}$. By Lemma \ref{split}, $E'|_{C_{s}}$ splits, a contradiction. Hence $ \mathrm{H}^{0}(R^{1}p_{*}\mathcal{G})=0$, and $ \mathrm{h}^{1}(p_{*}\mathcal{G})=p\neq 0$.

  Note that $v^{*}(E|_{X})$ is a modification of $E_{X, d}$. By Lemma \ref{split}, $E|_{C_{s}}\cong \mathcal{O}_{C_{s}}^{\oplus p}\oplus F_{C_{s}, d}$. By the same reason applied to sequence (\ref{GG}), 
  $$ \mathrm{H}^{0}(R^{1}p_{*}\underline{Hom}(G_{2}, G_{1}))= \mathrm{H}^{0}(R^{1}p_{*}\underline{Hom}(\mathcal{F}_{\widetilde{X}, d}(bF), \mathcal{O}_{\widetilde{X}}(aF-F)^{\oplus p}))= \mathrm{H}^{0}(R^{1}p_{*}\mathcal{G}(c)) $$
  must vanish. Hence $\mathrm{h}^{0}(R^{1}p_{*}\mathcal{G}(c))=0$ and $\mathrm{h}^{1}(p_{*}\mathcal{G}(c))=p$.
  \\
  \\
  \textbf{Step 5}: Show that $v^{*}(E|_{X})\cong v^{*}(E_{X, d})$. 

  Since $\mathrm{h}^{1}(p_{*}\mathcal{G})=p\neq 0$ and $p_{*}\mathcal{G}$ is a sheaf on $\mathbb{P}^{1}$, we must have $c=0$. Hence $a=b$. Since the restrictions of $v^{*}(E|_{X})$ and $\mathcal{E}_{\widetilde{X}, d}$ to $\Sigma$ are both trivial, we have $a=b=0$. Hence $v^{*}(E|_{X})$ and $\mathcal{E}_{\widetilde{X}, d}$ are both rigid objects in $\mathrm{Ext}^{1}(\mathcal{F}_{\widetilde{X}, d}, \mathcal{O}_{\widetilde{X}}(-F)^{\oplus p})$, we must have $v^{*}(E|_{X})\cong \mathcal{E}_{\widetilde{X}, d}$. Taking pushforward along $v$, we get
  $$E|_{X}\cong v_{*}(v^{*}E|_{X})\cong v_{*}(\mathcal{E}_{\widetilde{X}, d}) \cong E_{X, d}. $$
  
\end{proof}

\section{Main results}\label{section5}
Recall that we have chosen a pencil of quartic surfaces $L\subset \mathbb{P}^{34}$ that satisfies Condition \ref{condition1}, and $\widetilde{\mathbb{P}^{3}}=\widetilde{\mathbb{P}^{3}_{L}}$ is the total space of the pencil of quartics parametrized by $L$. In this section we prove the following main theorem of this paper. The proof follows a similar strategy as that of Proposition \ref{EXdunique}.

\begin{theorem}\label{main}
  On $\mathbb{P}^{3}$, there exists no exceptional bundles of rank $2d^{2}+1$ and degree $d$ for every $|d|\geq 4$. 
\end{theorem}

We first make the following observation.

\begin{lemma}\label{locally1}
  Let $E$ be an exceptional vector bundle on $\widetilde{\mathbb{P}^{3}}$ and $F$ be a vector bundle such that $E|_{X_{t}}\cong F|_{X_{t}}$ for all $t\in L$. Then $F$ is a modification of $E$ in the sense of Definition \ref{modification}.
\end{lemma}

\begin{proof}
  The lemma will follow from Lemma \ref{locally}. By assumption, condition (1) in Lemma \ref{locally} is satisfied. Hence it suffices to check condition (2) in Lemma \ref{locally}. Since $\widetilde{\mathbb{P}^{3}}=\mathrm{BL}_{C}\mathbb{P}^{3}$, we have
  $$\omega_{\widetilde{\mathbb{P}^{3}}}\cong \nu^{*}\omega_{\mathbb{P}^{3}}\otimes \mathcal{O}_{\widetilde{\mathbb{P}^{3}}}(Y)\cong \mathcal{O}_{\widetilde{\mathbb{P}^{3}}}(-4H+Y)\cong \mathcal{O}_{\widetilde{\mathbb{P}^{3}}}(-X_{t}), \forall t\in L.$$
  Since $E$ is exceptional, we have $\mathrm{Ext}^{2}(E, E(-X_{t}))\cong \mathrm{Ext}^{1}(E,E)^{*}=0$ by Serre Duality. Hence the following exact sequence
  $$0=\mathrm{Ext}^{1}(E,E) \longrightarrow \mathrm{Ext}^{1}(E, E|_{X_{t}}) \longrightarrow \mathrm{Ext}^{2}(E, E(-X_{t}))=0$$
  implies that $\mathrm{Ext}^{1}_{X_{t}}(E_{t}, E_{t})=0$ for all $t\in L$, condition (2) in Lemma \ref{locally} is satisfied.
\end{proof}

As an application, we prove Theorem 1.2 in \cite{Pol11} from a different perspective.

\begin{corollary}[Theorem 1.2, \cite{Pol11}]\label{Polcor}
  Let $E, F$ be exceptional bundles on $\mathbb{P}^{3}$ with the same rank and degree. If $\mathrm{Ext}^{1}_{\mathbb{P}^{3}}(E, E(4H))=0$, then $E\cong F$. 
\end{corollary}

\begin{proof}
By assumption, the following exact sequence
  $$0= \mathrm{Ext}^{1}_{\mathbb{P}^{3}}(E, E(4H)) \longrightarrow \mathrm{Ext}^{1}_{X}(E|_{X}, E|_{X}(4H)) \longrightarrow \mathrm{Ext}^{2}_{\mathbb{P}^{3}}(E, E)=0$$
  implies that $\mathrm{Ext}^{1}_{X}(E|_{X}, E|_{X}(-4H))=\mathrm{Ext}^{1}_{X}(E|_{X}, E|_{X}(4H))^{*}=0$. By Proposition \ref{Pol}, we have $E|_{X}\cong F|_{X}$ for all $X\in L$. By Lemma \ref{locally1}, $\nu^{*}F$ is a modification of $\nu^{*}(E)$. By Lemma \ref{simplerigid}, they are fiberwise simple. By Lemma \ref{fiberwisesimple}, the modification is trivial. 
\end{proof}

We construct a sheaf $\mathcal{E}_{d}$ on $\widetilde{\mathbb{P}^{3}}$ that is swept out by $E_{X, d}$ when $X$ varies in the pencil $L$, i.e. $\mathcal{E}_{d}|_{X}\cong E_{X, d}$ for all $[X]\in L$. One way to do this is to let $\mathcal{E}_{d}$ fit into the following exact sequence
\begin{equation}\label{releval}
0 \longrightarrow \mathcal{E}_{d} \longrightarrow \pi^{*}(\pi_{*}\mathcal{O}_{\widetilde{\mathbb{P}^{3}}}(dH)) \overset{ev}{\longrightarrow} \mathcal{O}_{\widetilde{\mathbb{P}^{3}}}(dH) \longrightarrow 0.
\end{equation}

Recall that $C$ is the base curve of $L\subset \mathbb{P}^{34}$ and $\widetilde{\mathbb{P}^{3}}=\mathrm{BL}_{C}\mathbb{P}^{3}$. Let $Y$ be the exceptional divisor of $\nu: \widetilde{\mathbb{P}^{3}} \longrightarrow \mathbb{P}^{3}$. Let $W_{n}:=  \mathrm{H}^{0}(C, \mathcal{O}_{C}(n))$ for all $n\in \mathbb{Z}$, $w_{n}=\mathrm{dim}W_{n}$, and $F$ be the fiber class of $\pi$. Write $d=4k+e$ for $k\in \mathbb{Z}$ and $0\leq e < 4$ uniquely.

\begin{lemma}\label{restriction}
  For $d\geq 0$, we have
  $$\mathcal{E}_{d}|_{Y}\cong \left[\bigoplus_{j=1}^{k} (\mathcal{O}_{Y}(jF)\otimes W_{d-4j})\right] \oplus \nu^{*}F_{C, d}, $$
  where $F$ is the fiber class of $\pi$, and $F_{C, d}$ is defined in (\ref{FCd}). 
\end{lemma}

\begin{proof}
  Since $C$ is a $(4,4)$-complete intersection, we have $Y\cong C\times L$. In $\mathrm{Pic}(\widetilde{\mathbb{P}^{3}})$, we have $Y+F=4H$. To compute $\pi_{*}\mathcal{O}_{\widetilde{\mathbb{P}^{3}}}(dH)$, we will apply $\pi_{*}$ to the following exact sequence on $\widetilde{\mathbb{P}^{3}}$:
  \begin{equation}
    0 \longrightarrow \mathcal{O}_{\widetilde{\mathbb{P}^{3}}}(dH-Y) \longrightarrow \mathcal{O}_{\widetilde{\mathbb{P}^{3}}}(dH) \longrightarrow \mathcal{O}_{Y}(dH) \longrightarrow 0.
  \end{equation}
  Let $X$ be any fiber of $\pi: \widetilde{\mathbb{P}^{3}} \rightarrow L$. Then $X$ has $ \mathrm{H}^{1}(X, \mathcal{O}_{X})=0$ and $\omega_{X}\cong \mathcal{O}_{X}$. By the Kodaira Vanishing Theorem and Serre duality, we have $ \mathrm{H}^{1}(X, \mathcal{O}_{X}(nH))=0$ for all $n\in \mathbb{Z}$. Hence we have
  $$R^{1}\pi_{*}\mathcal{O}_{\widetilde{\mathbb{P}^{3}}}(dH-Y)=R^{1}\pi_{*}\mathcal{O}_{\widetilde{\mathbb{P}^{3}}}((d-4)H+F)\cong R^{1}\pi_{*}\mathcal{O}_{\widetilde{\mathbb{P}^{3}}}((d-4)H)(1)=0.$$
   Then we have
  $$\pi_{*}\mathcal{O}_{Y}(dH)=\pi_{*}\mathcal{O}_{C\times L}(dH)= \mathcal{O}_{L}\otimes  \mathrm{H}^{0}(C, \mathcal{O}_{C}(dH))= \mathcal{O}_{L}\otimes W_{d}. $$
  Then we claim that for $d\geq 0$, we have
  \begin{eqnarray*}
    & &\pi_{*}\mathcal{O}_{\widetilde{\mathbb{P}^{3}}}(dH) \cong \bigoplus_{j=0}^{k} (\mathcal{O}_{L}(j)\otimes W_{d-4j}) \\
   & \cong & (\mathcal{O}_{L}(k)\otimes W_{e}) \oplus (\mathcal{O}_{L}(k-1)\otimes W_{e+4}) \oplus \cdots \oplus (\mathcal{O}_{L}(1)\otimes W_{d-4}) \oplus (\mathcal{O}_{L}\otimes W_{d}).
    \end{eqnarray*}
  We use induction on $k$. By the projection formula, we have
  $$\pi_{*}\mathcal{O}_{\widetilde{\mathbb{P}^{3}}}(dH-Y)=\pi_{*}\mathcal{O}_{\widetilde{\mathbb{P}^{3}}}((d-4)H+F)=\pi_{*}\mathcal{O}_{\widetilde{\mathbb{P}^{3}}}((d-4)H)(1).$$
  When $k=0$, namely $0\leq d < 4$, we have $ \mathrm{H}^{0}(X, \mathcal{O}_{X}((d-4)H))=0$ for every fiber $X$ of $\pi$. Hence $\pi_{*}\mathcal{O}_{\widetilde{\mathbb{P}^{3}}}(dH-Y)=0$. 
  Therefore $\pi_{*}\mathcal{O}_{\widetilde{\mathbb{P}^{3}}}(dH) \cong \pi_{*}\mathcal{O}_{Y}(dH)\cong \mathcal{O}_{L}\otimes W_{d}$, the claim holds in this case.
  Now suppose the claim holds for $(k-1)$. Then
  $$\pi_{*}\mathcal{O}_{\widetilde{\mathbb{P}^{3}}}(dH-Y)\cong \left[\bigoplus_{j=0}^{k-1}(\mathcal{O}_{L}(j)\otimes W_{d-4-4j})\right] \otimes \mathcal{O}_{L}(1) = \bigoplus_{j=1}^{k}(\mathcal{O}_{L}(j)\otimes W_{d-4j}). $$
  Hence $\pi_{*}\mathcal{O}_{\widetilde{\mathbb{P}^{3}}}(dH)$ fits into the following exact sequence on $L$:
  $$0 \longrightarrow  \bigoplus_{j=1}^{k}(\mathcal{O}_{L}(j)\otimes W_{d-4j}) \longrightarrow \pi_{*}\mathcal{O}_{\widetilde{\mathbb{P}^{3}}}(dH) \longrightarrow \mathcal{O}_{L}\otimes W_{d} \longrightarrow 0. $$
  This sequence splits, since
  $$\mathrm{Ext}^{1}_{L}(\mathcal{O}_{L}\otimes W_{d},  \bigoplus_{j=1}^{k}(\mathcal{O}_{L}(j)\otimes W_{d-4j}))=0. $$
  Hence the claim is proved.

  Restricting the sequence (\ref{releval}) to $Y\cong C\times L$, we get
  $$0 \longrightarrow \mathcal{E}_{d}|_{C\times L} \longrightarrow \bigoplus_{j=0}^{k} (\mathcal{O}_{C\times L}(jF)\otimes W_{d-4j}) \overset{ev|_{C\times L}}{\longrightarrow} \mathcal{O}_{C\times L}(dH) \longrightarrow 0. $$
  Recall that $F_{C, d}$ fits into the following exact sequence by construction (\ref{FCd}):
  $$0 \longrightarrow F_{C, d} \longrightarrow \mathcal{O}_{C}\otimes W_{d} \overset{ev_{C}}{\longrightarrow} \mathcal{O}_{C}(dH) \longrightarrow 0.$$
  On the summand $\mathcal{O}_{C\times L}\otimes W_{d}$, the map $ev|_{C\times L}$ is the pullback of $ev_{C}$ along $\nu$. On other summands, $ev|_{C\times L}$ vanishes. Hence we have
  $$\mathcal{E}_{d}|_{Y}\cong \left[\bigoplus_{j=1}^{k} (\mathcal{O}_{Y}(jF)\otimes W_{d-4j})\right] \oplus \nu^{*}F_{C, d}. $$
\end{proof}

Now we prove Theorem \ref{main}.

\begin{proof}[Proof of Theorem \ref{main}]
  Suppose $E$ is an exceptional bundle on $\mathbb{P}^{3}$ with rank $2d^{2}+1$ and degree $-d$. We assume $d\geq 0$, the case where $d<0$ can be proved by taking dual. By Proposition \ref{EXdunique},
  $$\nu^{*}(E)|_{X_{t}}\cong E|_{X_{t}}\cong E_{X_{t}, d}\cong \mathcal{E}_{d}|_{X_{t}}.$$
  Since $E$ is exceptional on $\mathbb{P}^{3}$, $\nu^{*}(E)$ is exceptional on $\widetilde{\mathbb{P}^{3}}$. By Lemma \ref{locally1},
  $\nu^{*}(E)$ is a modification of $\mathcal{E}_{d}$ (Definition \ref{modification}). By Corollary \ref{Edrigid}, $\mathcal{E}_{d}|_{X}$ is simple and rigid for all fibers $X$ of $\pi$. By Lemma \ref{fiberwisesimple}, $\nu^{*}E$ is a trivial modification of $\mathcal{E}_{d}$. In particular, since $\nu^{*}(E)|_{Y}$ is trivial, we must have $\mathcal{E}_{d}|_{Y}\cong \mathcal{O}_{Y}(mF)^{\oplus q}$ for some $m, q\in \mathbb{Z}$. However, by Lemma \ref{restriction}, we have
  $$\mathcal{E}_{d}|_{Y}\cong \left[\bigoplus_{j=1}^{k} (\mathcal{O}_{Y}(jF)\otimes W_{d-4j})\right] \oplus \nu^{*}F_{C, d}. $$
  Recall that $W_{n}= \mathrm{H}^{0}(C, \mathcal{O}_{C}(nH))$. When $d\geq 4$, $W_{d-4j}\neq 0$ for $1\leq j\leq k$. Hence $\mathcal{E}_{d}$ is not of the form $\mathcal{O}_{Y}(mF)^{\oplus q}$, a contradiction.
  
\end{proof}

\section{Another obstruction}\label{section6}

In this section we discuss a different phenomenon from that of Theorem \ref{main} which obstructs the existence of exceptional bundles. The following theorem is a consequence of this different obstruction. 

\begin{theorem}\label{rank27}
On $\mathbb{P}^{3}$, there exists no exceptional bundles of rank $27$ and degree $11$. 
\end{theorem}

In this section, for any smooth quartic surface $X\subset \mathbb{P}^{3}$ with Picard rank 1, we fix the Mukai vector $v=(27, 11H, 9)\in  \mathrm{H}^{*}_{alg}(X)$ and let $E_{X}\in M_{X, H}(v)$. Suppose $E$ is an exceptional bundle with rank $27$ and degree $11$. Then for such an $X$ we have $E|_{X}\cong E_{X}$ \cite{Zub90}. Let $\{X_{t}\subset \mathbb{P}^{3}\}_{t\in L}, L\subset \mathbb{P}^{34}$ be a general pencil of quartic surfaces and $C\subset \mathbb{P}^{3}$ be its base curve. If such an exceptional bundle $E$ exists, then
$$E_{X_{t}}|_{C} \cong E|_{C} \cong E_{X_{t'}}|_{C}$$
for any $t, t'\in L$. However, in fact this is not true: the following phenomenon obstructs the existence of exceptional bundles of rank $27$ and degree $11$.
\begin{phenomenon}[Corollary \ref{nonisomorphic}]\label{deform}
  For general $t, t'\in L$, $E_{X_{t}}|_{C}\not\cong E_{X_{t'}}|_{C}$.
  \end{phenomenon}

Note that Phenomenon \ref{deform} is different from the the reason for Theorem \ref{main}. In Theorem \ref{main}, $E_{X_{t}}|_{C}\cong E_{X_{t'}}|_{C}$ for any $t, t'\in L$. The difficulty there is to rule out modifications (Definition \ref{modification}) of $E_{X_{t}}|_{C}$, since they are not simple. However, in the case of rank $27$ and degree $11$, we have $E_{X_{t}}|_{C}$ is simple for general $t\in L$. The difficulty here is to show $E_{X_{t}}|_{C}$ and $E_{X_{t'}}|_{C}$ are non-isomorphic for general $t, t'\in L$. This also shows that the condition $\mathrm{Ext}^{1}_{X}(E, E(4H))=0$ in \cite{Pol11} is not a sufficient condition for the existence of exceptional bundles.

We first describe $E_{X}$ for a smooth quartic surface $X$ with Picard rank 1. By the algorithm in \cite{Liu22}, $E_{X}$ can be constructed as follows. Let $M=T_{\mathbb{P}^{3}}(-H)$. Let $N_{X}\in \mathcal{D}^{b}(X)$ fits into the following exact triangle
$$ \mathcal{O}_{X}[1]\otimes  \mathrm{H}^{0}(X, \mathcal{O}_{X}(H)) \longrightarrow N_{X} \longrightarrow \mathcal{O}_{X}(H),$$
where the extension is described below. Note that $\mathrm{Ext}_{X}^{1}(\mathcal{O}_{X}(H),\mathcal{O}_{X}[1])\cong  \mathrm{H}^{0}(X, \mathcal{O}_{X}(H))^{*}$. Hence we may take the coevaluation map $\mathcal{O}_{X}(H) \overset{coev}{\longrightarrow} \mathcal{O}_{X}[2]\otimes  \mathrm{H}^{0}(X, \mathcal{O}_{X}(H))$ and let $N_{X}$ be the cone. Then $E_{X}$ can be constructed as the canonical extension that fits into the following exact triangle
\begin{equation}\label{construction1}
  M|_{X}\otimes \mathrm{Ext}_{X}^{1}(N_{X},M|_{X})^{*} \longrightarrow E_{X} \longrightarrow N_{X}.
  \end{equation}
We have a canonical identification $\mathrm{Ext}^{1}(N_{X}, M|_{X})^{*}\cong \mathrm{H}^{0}(X, \mathcal{O}_{X}(2H))$.
For simplicity, we set the following notation.
\begin{notation}\label{Vd}
  For any quartic surface $X\subset \mathbb{P}^{3}$ and $d\in \mathbb{Z}$, let $V_{X, d}:= \mathrm{H}^{0}(X, \mathcal{O}_{X}(dH))$. For $d\leq 3$, they are fixed quotients of $ \mathrm{H}^{0}(\mathbb{P}^{3}, \mathcal{O}_{\mathbb{P}^{3}}(dH))$, we denote them by $V_{d}$ in this case.
\end{notation}

Using Notation \ref{Vd}, the long exact sequence of cohomology sheaves of the sequence (\ref{construction1}) is
\begin{equation}\label{construction2}
  0 \longrightarrow \mathcal{O}_{X}\otimes V_{1} \overset{f}{\longrightarrow} M|_{X} \otimes V_{2} \longrightarrow E_{X} \longrightarrow \mathcal{O}_{X}(H) \longrightarrow 0.
\end{equation}
The map $f$ is given by the composition of the following maps
$$\mathcal{O}_{X}\otimes V_{1} \overset{coev\otimes \mathrm{id}}{\longrightarrow} M|_{X}\otimes V_{1}\otimes V_{1} \overset{\mathrm{id}\otimes m}{\longrightarrow} M\otimes V_{2}, $$
where $m: V_{1}\otimes V_{1} \longrightarrow V_{2}$ is the multiplication map. Note that using the same description, $f$ can be defined on $\mathbb{P}^{3}$. We let $F$ be the cokernel:
\begin{equation}\label{FF}
  0 \longrightarrow \mathcal{O}_{\mathbb{P}^{3}}\otimes V_{1} \overset{f}{\longrightarrow} M \otimes V_{2} \longrightarrow F \longrightarrow 0.
  \end{equation}
Note that $\mathrm{Ext}^{1}_{\mathbb{P}^{3}}(\mathcal{O}_{\mathbb{P}^{3}}(H), F)=0$, hence the same construction cannot produce a simple bundle on $\mathbb{P}^{3}$. However, on each quartic $X$ there are non-trivial extensions.

\begin{lemma}\label{uniqueextension}
We have $\mathrm{Ext}_{X}^{1}(\mathcal{O}_{X}(H), F|_{X})=\mathbb{C}$ for \emph{all} $X\in L$. 
\end{lemma}

\begin{proof}
  Applying $\mathrm{Hom}_{X}(\mathcal{O}_{X}(H),-)$ to sequence (\ref{FF}), we get the following exact sequence
  \begin{multline*}
    \mathrm{Ext}_{X}^{1}(\mathcal{O}_{X}(H), M_{X}\otimes V_{2}) \longrightarrow \mathrm{Ext}_{X}^{1}(\mathcal{O}_{X}(H), F|_{X})\\
    \longrightarrow \mathrm{Ext}_{X}^{2}(\mathcal{O}_{X}(H), \mathcal{O}_{X}\otimes V_{1}) \overset{g}{\longrightarrow} \mathrm{Ext}_{X}^{2}(\mathcal{O}_{X}(H), M|_{X}\otimes V_{2}).
    \end{multline*}
  Computing each term, we have
  \begin{eqnarray*}
 &  & \mathrm{Ext}_{X}^{1}( \mathcal{O}_{X}(H), M|_{X}\otimes V_{2})=0, \\
  &  &\mathrm{Ext}_{X}^{2}(\mathcal{O}_{X}(H), \mathcal{O}_{X}\otimes V_{1})=V_{1}^{*}\otimes V_{1}, \\
& &    \mathrm{Ext}_{X}^{2}(\mathcal{O}_{X}(H), M|_{X}\otimes V_{2})=(\wedge^{2}V_{1})^{*}\otimes V_{2}.
  \end{eqnarray*}
  The map $g$ can be explicitly described as follows. Let $g_{1}: V_{1}^{*} \longrightarrow (\wedge^{2}V_{1})^{*}\otimes V_{1}$ be the adjoint map of the canonical projection $V_{1}^{*}\otimes V_{1}^{*} \longrightarrow (\wedge^{2}V_{1})^{*}$. Then $g$ is the composition of the following maps:
  $$V_{1}^{*}\otimes V_{1} \overset{g_{1}\otimes \mathrm{id}}{\longrightarrow} (\wedge^{2}V_{1})^{*}\otimes V_{1}\otimes V_{1} \overset{\mathrm{id}\otimes m}{\longrightarrow} (\wedge^{2}V_{1})\otimes V_{2}. $$
  Let $x_{0}, \cdots, x_{3}$ be coordinates on $\mathbb{P}^{3}$, then by explicit computations we have
  $$\mathrm{Ext}^{1}(\mathcal{O}_{X}(H), F|_{X})\cong \mathrm{ker}(g)=\mathbb{C}\delta_{1}:=\mathbb{C}\left(\sum_{i=0}^{3}x_{i}^{*}\otimes x_{i}\right)\in V_{1}^{*}\otimes V_{1}. $$
\end{proof}

We define the vector bundle $E_{X}$ on \emph{any} $X\in L$ to be the unique non-split extension
\begin{equation}\label{EX}
  0 \longrightarrow F|_{X} \longrightarrow E_{X} \longrightarrow \mathcal{O}_{X}(H) \longrightarrow 0.
  \end{equation}

The author wonders but does not know whether $F|_{C}$ is stable. However the following proposition is sufficient for our purpose.
\begin{proposition}\label{FC}
The bundle $F|_{C}$ is simple.
\end{proposition}

We will show Proposition \ref{FC} by proving the following two lemmas.

\begin{lemma}\label{FP3}
The bundle $F$ is simple on $\mathbb{P}^{3}$. 
\end{lemma}

\begin{proof}
  Applying $\mathrm{Hom}_{\mathbb{P}^{3}}(F,-)$ to (\ref{FF}), we have
  $$0= \mathrm{Hom}_{\mathbb{P}^{3}}(F,M\otimes V_{2}) \longrightarrow \mathrm{Hom}_{\mathbb{P}^{3}}(F,F) \longrightarrow \mathrm{Ext}_{\mathbb{P}^{3}}^{1}(F, \mathcal{O}_{\mathbb{P}^{3}}\otimes V_{1}) \overset{r}{\longrightarrow} \mathrm{Ext}_{\mathbb{P}^{3}}^{1}(F, M\otimes V_{2}). $$
  Hence $\mathrm{Hom}_{\mathbb{P}^{3}}(F,F)=\mathrm{ker}(r)$. By applying $\mathrm{Hom}_{\mathbb{P}^{3}}(-,\mathcal{O}_{\mathbb{P}^{3}})$ and $\mathrm{Hom}_{\mathbb{P}^{3}}(-,M)$ to (\ref{FF}), we see that
  $$\mathrm{Ext}_{\mathbb{P}^{3}}^{1}(F,\mathcal{O}_{\mathbb{P}^{3}})=V_{1}^{*}, \mathrm{Ext}_{\mathbb{P}^{3}}^{1}(F,M)=V_{1}^{*}\otimes V_{1}^{*}.$$
  The map $r$ is the composition of the following maps
  $$V_{1}^{*}\otimes V_{1}=V_{1}^{*}\otimes \mathbb{C}\otimes V_{1} \overset{\mathrm{id}\otimes \mathrm{tr}^{*} \otimes \mathrm{id}}{\longrightarrow} V_{1}^{*}\otimes V_{1}^{*}\otimes V_{1}\otimes V_{1} \overset{\mathrm{id}\otimes m}{\longrightarrow} V_{1}^{*}\otimes V_{1}^{*}\otimes V_{2},  $$
  where $\mathrm{tr}^{*}$ is the dual of the trace map $V_{1}^{*}\otimes V_{1} \rightarrow \mathbb{C}$. By explicit coordinates computations, we see that
  $$\mathrm{ker}(r)=\mathbb{C}\left(\sum_{i=0}^{3}x_{i}^{*}\otimes x_{i} \right)\in V_{1}^{*}\otimes V_{1}. $$
\end{proof}

\begin{lemma}\label{FX}
The bundle $F|_{X}$ is simple on $X$. 
\end{lemma}

\begin{proof}
  Applying $\mathrm{Hom}_{\mathbb{P}^{3}}(F,-)$ to the exact sequence
  $$0 \longrightarrow F(-4H) \longrightarrow F \longrightarrow F|_{X} \longrightarrow 0,$$
  we have
  $$0=\mathrm{Hom}_{\mathbb{P}^{3}}(F,F(-4H)) \longrightarrow \mathrm{Hom}_{\mathbb{P}^{3}}(F,F) \longrightarrow \mathrm{Hom}_{\mathbb{P}^{3}}(F, F|_{X}) \longrightarrow \mathrm{Ext}_{\mathbb{P}^{3}}^{1}(F,F(-4H)).$$
  By Lemma \ref{FP3}, $\mathrm{Hom}_{\mathbb{P}^{3}}(F,F)=\mathbb{C}$, hence it suffices to show $\mathrm{Ext}_{\mathbb{P}^{3}}^{1}(F,F(-4H))=0$.
  Applying $\mathrm{Hom}_{\mathbb{P}^{3}}(F, -$) to (\ref{FF}) twisted by $\mathcal{O}_{\mathbb{P}^{3}}(-4H)$, we have
  $$\mathrm{Ext}_{\mathbb{P}^{3}}^{1}(F,M(-4H)\otimes V_{2}) \longrightarrow \mathrm{Ext}_{\mathbb{P}^{3}}^{1}(F, F(-4H)) \longrightarrow \mathrm{Ext}_{\mathbb{P}^{3}}^{2}(F, \mathcal{O}_{\mathbb{P}^{3}}(-4H)\otimes V_{1})=0. $$
  Applying $\mathrm{Hom}_{\mathbb{P}^{3}}(-,M(-4H))$ to (\ref{FF}), we have
  $$0=\mathrm{Hom}_{\mathbb{P}^{3}}(\mathcal{O}_{\mathbb{P}^{3}}\otimes V_{1}, M(-4H)) \longrightarrow \mathrm{Ext}_{\mathbb{P}^{3}}^{1}(F, M(-4H)) \longrightarrow \mathrm{Ext}_{\mathbb{P}^{3}}^{1}(M\otimes V_{2}, M(-4H))=0. $$
  Hence $\mathrm{Ext}_{\mathbb{P}^{3}}^{1}(F,M(-4H))=\mathrm{Ext}_{\mathbb{P}^{3}}^{1}(F,F(-4H))=0$, and $F|_{X}$ is simple for all $X$. 
  
\end{proof}

\begin{proof}[Proof of Proposition \ref{FC}]
  Applying $\mathrm{Hom}_{X}(F|_{X}, -$) to the exact sequence
  $$0 \longrightarrow F|_{X}(-4H) \longrightarrow F|_{X} \longrightarrow F|_{C} \longrightarrow 0,$$
  we have
\begin{multline*}
  0=\mathrm{Hom}_{X}(F|_{X},F|_{X}(-4H)) \longrightarrow \mathrm{Hom}_{X}(F|_{X}, F|_{X})\\
  \longrightarrow \mathrm{Hom}_{X}(F|_{X}, F|_{C}) \longrightarrow \mathrm{Ext}_{X}^{1}(F|_{X}, F|_{X}(-4H)).
\end{multline*}
  By Lemma \ref{FX}, it suffices to show $\mathrm{Ext}_{X}^{1}(F|_{X},F|_{X}(-4H))=0$. Applying $\mathrm{Hom}_{X}(F|_{X},-)$ to the restriction of (\ref{FF}) to $X$, we have
  \begin{multline*}
    \mathrm{Ext}_{X}^{1}(F|_{X}, M|_{X}(-4H)\otimes V_{2}) \longrightarrow \mathrm{Ext}_{X}^{1}(F|_{X}, F|_{X}(-4H))\\
    \longrightarrow \mathrm{Ext}_{X}^{2}(F, \mathcal{O}_{X}(-4H)\otimes V_{1}) \overset{u}{\longrightarrow} \mathrm{Ext}_{X}^{2}(F,M|_{X}(-4H)\otimes V_{2}).
  \end{multline*}
    Applying $\mathrm{Hom}_{X}(-,M|_{X}(-4H))$ to the restriction of (\ref{FF}) to $X$, we have
    \begin{multline*}
      0=\mathrm{Hom}_{X}(\mathcal{O}_{X}\otimes V_{1}, M|_{X}(-4H)) \longrightarrow \mathrm{Ext}_{X}^{1}(F|_{X}, M|_{X}(-4H))\\
      \longrightarrow \mathrm{Ext}_{X}^{1}(M|_{X}\otimes V_{2}, M|_{X}(-4H))=0.
      \end{multline*}
    Hence $\mathrm{Ext}_{X}^{1}(F|_{X}, M|_{X}(-4H))=0$. To show $\mathrm{Ext}_{X}^{1}(F|_{X}, F|_{X}(-4H))=0$, it suffices to show the map $u$ is injective, equivalently, to show $u^{*}$ is surjective. Consider the following commutative diagram
    \[\begin{tikzcd}
        \mathrm{Hom}_{X}(M|_{X}(-4H),M|_{X}\otimes V_{2})\otimes V_{2}^{*}\arrow[r, "v\otimes \mathrm{id}_{V_{2}}"]\arrow[d] & \mathrm{Hom}_{X}(\mathcal{O}_{X}(-4H),M|_{X}\otimes V_{2})\otimes V_{1}^{*} \arrow[d] \\
        \mathrm{Hom}_{X}(M|_{X}(-4H),F)\otimes V_{2}^{*} \arrow[r, "u^{*}"]\arrow[d] & \mathrm{Hom}_{X}(\mathcal{O}_{X}(-4H), F)\otimes V_{1}^{*}\arrow[d] \\
        \mathrm{Ext}_{X}^{1}(M|_{X}(-4H),\mathcal{O}_{X}\otimes V_{1})\otimes V_{2}^{*}=0 & \mathrm{Ext}_{X}^{1}(\mathcal{O}_{X}(-4H), \mathcal{O}_{X}\otimes V_{1})\otimes V_{1}^{*}=0,
      \end{tikzcd}\]
    where the columns are exact. Hence it suffices to show the surjectivity of the map 
    $$v: \mathrm{Hom}_{X}(M|_{X}(-4H), M|_{X})\otimes V_{2} \rightarrow \mathrm{Hom}_{X}(\mathcal{O}_{X}(-4H), M|_{X})\otimes V_{1}^{*},$$
     which is described as follows. Applying $\mathrm{Hom}_{X}(-,M|_{X})$ to the Euler sequence
    $$0 \longrightarrow \mathcal{O}_{X}(-5H) \longrightarrow \mathcal{O}_{X}(-4H)\otimes V_{1}^{*} \overset{ev}{\longrightarrow} M|_{X}(-4H) \longrightarrow 0, $$
    we get $ev^{*}: \mathrm{Hom}_{X}(M|_{X}(-4H),M|_{X}) \longrightarrow \mathrm{Hom}_{X}(\mathcal{O}_{X}(-4H), M|_{X})\otimes V_{1}$. Let $e_{2}^{*}: V_{1}\otimes V_{2}^{*} \rightarrow V_{1}^{*}$ be the adjoint map of the comultiplication map $V_{2}^{*} \rightarrow V_{1}^{*}\otimes V_{1}^{*}$. Then $v$ is the composition of the following maps
    \begin{multline*}
      \mathrm{Hom}_{X}(M|_{X}(-4H),M|_{X})\otimes V_{2}^{*} \overset{ev\otimes \mathrm{id}_{V_{2}^{*}}}{\longrightarrow} \mathrm{Hom}_{X}(\mathcal{O}_{X}(-4H), M|_{X})\otimes V_{2}^{*} \otimes V_{1}\\
      \overset{\mathrm{id}\otimes e_{2}^{*}}{\longrightarrow} \mathrm{Hom}_{X}(\mathcal{O}_{X}(-4H), M|_{X})\otimes V_{1}^{*}.
      \end{multline*}
    By an explicit computation in coordinates, $v$ is surjective.
\end{proof}

We also need the following fact.
\begin{proposition}\label{uniqueinclusion}
For any $X\in L$, we have $\mathrm{Hom}_{C}(F|_{C}, E_{X}|_{C})=\mathbb{C}$. 
\end{proposition}

\begin{proof}
  Applying $\mathrm{Hom}_{C}(- , E_{X}|_{C})$ to the sequence (\ref{EX}), we get
  \begin{equation}\label{FE}
    0 \longrightarrow \mathrm{Hom}_{C}(F|_{C}, E_{X}|_{C}) \longrightarrow \mathrm{Hom}_{C}(M|_{C}\otimes V_{2}, E_{X}|_{C}) \overset{h}{\longrightarrow} \mathrm{Hom}(\mathcal{O}_{C}\otimes V_{1}, E_{X}|_{C}).
    \end{equation}
  We compute each term in the following.

  To compute $\mathrm{Hom}_{C}(\mathcal{O}_{C}, E_{X}|_{C})$, we use the sequence (\ref{construction1}). Applying $ \mathrm{H}^{0}(X, -)$ to (\ref{construction1}), we get
  $$V_{1}= \mathrm{H}^{-1}(X, N_{X}) \overset{h_{1}}{\longrightarrow}  \mathrm{H}^{0}(X, M|_{X}\otimes V_{2}) \overset{p_{1}}{\longrightarrow}  \mathrm{H}^{0}(X, E_{X}) \longrightarrow  \mathrm{H}^{0}(X, N_{X})=0. $$
  We have nutural identification $ \mathrm{H}^{0}(X, M|_{X})=V_{1}^{*}$, and the map $h_{1}$ is the adjoint of the multiplication map $m: V_{1}\otimes V_{1} \longrightarrow V_{2}$. Hence $ \mathrm{H}^{0}(X, E_{X})=\mathrm{coker}(h_{1})$. Now we apply $ \mathrm{H}^{0}(X, -)$ to the sequnce
  $$0 \longrightarrow E_{X}(-4H) \longrightarrow E_{X} \longrightarrow E_{X}|_{C} \longrightarrow 0,  $$
  we get
  $$0= \mathrm{H}^{0}(X, E_{X}(-4H)) \longrightarrow  \mathrm{H}^{0}(X, E_{X}) \longrightarrow  \mathrm{H}^{0}(X, E_{X}|_{C}) \longrightarrow  \mathrm{H}^{1}(X, E_{X}(-4H)). $$
  We will show $ \mathrm{H}^{1}(X, E_{X}(-4H))=0$ by applying $ \mathrm{H}^{0}(X, -)$ to
  $$0 \longrightarrow M|_{X}(-4H)\otimes V_{2} \longrightarrow E_{X}(-4H) \longrightarrow N_{X}(-4H) \longrightarrow 0.$$
Let $K_{1}$ be the cokernel of the comultiplication $V_{X, 5}^{*} \longrightarrow V_{X, 4}^{*}\otimes V_{1}^{*}$. Then $ \mathrm{H}^{2}(X, M|_{X}(-4H))=K_{1}$. We also have $ \mathrm{H}^{1}(X, M|_{X}(-4H))=0$ and $ \mathrm{H}^{1}(X, N_{X})=V_{X, 4}^{*}\otimes V_{1}$. Hence $ \mathrm{H}^{1}(X, E_{X}(-4H))$ can be idetified with the kernel of the following composition of maps
$$h_{2}: V_{X, 4}^{*}\otimes V_{1} \overset{\mathrm{id}\otimes e_{2}}{\longrightarrow} V_{X, 4}^{*}\otimes V_{1}^{*}\otimes V_{2} \longrightarrow K_{1}\otimes V_{2},$$
where $e_{2}$ is the adjoint of the multiplication map $V_{1}\otimes V_{1} \longrightarrow V_{2}$. By an explicit computation in coordinates, we see that $h_{2}$ is injective. Hence $ \mathrm{H}^{1}(X, E_{X}(-4H))=0$ and we have $ \mathrm{H}^{0}(C, E_{X}|_{C})=\mathrm{coker}(h_{1})$.

To compute $\mathrm{Hom}_{C}(M|_{C}, E_{X}|_{C})$, we first compute $\mathrm{Hom}_{X}(M|_{X}, E_{X})$. Applying $\mathrm{Hom}_{X}(M,-)$ to (\ref{construction1}), we have
$$0 \longrightarrow \mathrm{Hom}_{X}(M|_{X}, M|_{X}\otimes V_{2}) \longrightarrow \mathrm{Hom}_{X}(M|_{X},E_{X}) \longrightarrow \mathrm{Hom}_{X}(M|_{X}, N_{X})=0. $$
Hence $\mathrm{Hom}_{X}(M, E_{X})=V_{2}$. Applying $\mathrm{Hom}|_{X}(M|_{X},-)$ to the sequence
$$0 \longrightarrow E_{X}(-4H) \longrightarrow E_{X} \longrightarrow E_{X}|_{C} \longrightarrow 0, $$
we have
\begin{multline*}
  0=\mathrm{Hom}_{X}(M|_{X}, E_{X}(-4H)) \longrightarrow \mathrm{Hom}_{X}(M|_{X}, E_{X})\\
  \longrightarrow \mathrm{Hom}_{C}(M|_{C}, E_{X}|_{C}) \longrightarrow \mathrm{Ext}_{X}^{1}(M|_{X}, E_{X}(-4H)).
  \end{multline*}
We will show $\mathrm{Ext}_{X}^{1}( M|_{X},E_{X}(-4H))=0$ by applying $\mathrm{Hom}_{X}(M|_{X}, -)$ to the sequence (\ref{construction1}) twisted by $\mathcal{O}_{X}(-4H)$. We get
\begin{multline*}
  \mathrm{Ext}_{X}^{1}(M,M(-4H)\otimes V_{2}) \longrightarrow \mathrm{Ext}_{X}^{1}(M, E_{X}(-4H))\\
  \longrightarrow \mathrm{Ext}_{X}^{1}(M, N_{X}(-4H)) \overset{e_{3}}{\longrightarrow} \mathrm{Ext}_{X}^{2}( M, M(-4H)\otimes V_{2}).
  \end{multline*}
Note that $e_{3}$ is exactly the map $v^{*}$, where $v$ is defined in the proof of Lemma \ref{FC}. Hence $e_{3}$ is injective. We have $\mathrm{Ext}_{X}^{1}(M|_{X}, E_{X}(-4H))=0$ and $\mathrm{Hom}_{C}(M|_{C}, E_{X}|_{C})=V_{2}$ for any $X$. 

Getting back to the sequence (\ref{FE}), the map $h$ is the composition of the following maps
$$V_{2}^{*}\otimes V_{2} \overset{m^{*}\otimes \mathrm{id}}{\longrightarrow} V_{1}^{*}\otimes V_{1}^{*}\otimes V_{2} \overset{\mathrm{id}\otimes p_{1}}{\longrightarrow} \mathrm{coker}(h_{1}), $$
where $m^{*}$ is the comultiplication map $V_{2}^{*} \longrightarrow V_{1}^{*}\otimes V_{1}^{*}$. By an explicit computation in coordinates, the kernel of $h$ is the identity element in $V_{2}^{*}\otimes V_{2}$. Hence $\mathrm{Hom}_{C}(F|_{C}, E_X|_{C})=\mathbb{C}$. 
  
\end{proof}

The following lemma allows us to prove certain bundles are non-isomorphic.
\begin{lemma}\label{scalar}
  On $C$, let
  $$\delta_{i}=[0 \longrightarrow F|_{C} \longrightarrow E_{i} \longrightarrow \mathcal{O}_{C}(H) \longrightarrow 0]\in \mathrm{Ext}^{1}_{C}(\mathcal{O}_{C}(H), F|_{C}), i=1,2 $$
  be extension classes such that $\mathrm{Hom}_{C}(F|_{C}, E_{i})=\mathbb{C}, i=1,2$. If $E_{1}\cong E_{2}$, then $\delta_{1}=\lambda \delta_{2}$ for some scalar $\lambda\in \mathbb{C}^{*}$.
\end{lemma}

\begin{proof}
  Let $f: E_{1} \rightarrow E_{2}$ be an isomorphism. Then the composition
  $$F|_{C}\subset E_{1} \overset{f}{\longrightarrow} E_{2}$$
  is non-zero. Since $\mathrm{Hom}(F|_{C}, E_{2})=\mathbb{C}$, we have the following commutative diagram
  \[
  \begin{tikzcd}
    0\arrow[r] & F|_{C} \arrow[r]\arrow[d, "f_{1}"] & E_{1} \arrow[r]\arrow[d, "f"] & \mathcal{O}_{C}(H) \arrow[r]\arrow[d, "f_{2}"] & 0 \\
    0\arrow[r] & F|_{C} \arrow[r] & E_{2} \arrow[r] & \mathcal{O}_{C}(H) \arrow[r] & 0.  
  \end{tikzcd}\]
By Proposition \ref{FC}, $F|_{C}$ and $\mathcal{O}_{C}(H)$ are both simple. Hence $f_{1}, f_{2}$ are non-zero scalar multiplications. 
\end{proof}

The following observation plays a central role in the proof of Theorem \ref{rank27}.

\begin{proposition}\label{nonconstant}
  For distinct $t_{i}\in L, i=1,2$, the restriction maps
  $$\mathrm{res}_{t_{i}}: \mathrm{Ext}^{1}_{X_{t_{i}}}(\mathcal{O}_{X_{t_{i}}}(H), F_{X_{t_{i}}}) \longrightarrow \mathrm{Ext}_{C}^{1}(\mathcal{O}_{C}(H), F_{C}), i=1,2 $$
  have distinct images, as subspaces of the fixed space $\mathrm{Ext}_{C}^{1}(\mathcal{O}_{C}(H), F_{C})$.
\end{proposition}

\begin{proof}
  Recall that $\widetilde{\mathbb{P}^{3}}=\mathrm{BL}_{C}\mathbb{P}^{3}$, $\nu: \widetilde{\mathbb{P}^{3}} \rightarrow \mathbb{P}^{3}$ is the projection and $\pi: \widetilde{\mathbb{P}^{3}} \rightarrow L$ is the $K3$ fibration. Let $Y\cong C\times L$ be the exceptional divisor of $\nu$. The pushforward $\pi_{*}$ on the restriction map $\nu^{*}F \longrightarrow \nu^{*}F|_{Y}$ induces the map
  $$\mathrm{res}: R^{1}\pi_{*}\underline{Hom}_{\widetilde{\mathbb{P}^{3}}}(\mathcal{O}_{\widetilde{\mathbb{P}^{3}}}(H), \nu^{*}F) \longrightarrow  R^{1}\pi_{*}\underline{Hom}_{\widetilde{\mathbb{P}^{3}}}(\mathcal{O}_{\widetilde{\mathbb{P}^{3}}}(H), \nu^{*}F|_{Y})\cong \mathrm{Ext}_{C}^{1}(\mathcal{O}_{C}(H), F_{C})\otimes \mathcal{O}_{L}. $$
  The fiber of $\mathrm{res}$ over $t\in L$ is the map
  $$\mathrm{res}_{t}: \mathrm{Ext}^{1}_{X_{t}}(\mathcal{O}_{X_{t}}(H), F_{X_{t}}) \longrightarrow \mathrm{Ext}_{C}^{1}(\mathcal{O}_{C}(H), F_{C}). $$
  Relativize the proof of Lemma \ref{uniqueextension} by using the relative Serre Duality, we have
  $$R^{1}\pi_{*}\underline{Hom}_{\widetilde{\mathbb{P}^{3}}}(\mathcal{O}_{\widetilde{\mathbb{P}^{3}}}(H), \nu^{*}F) \cong \mathcal{O}_{L}(-1).$$
  Hence the restrictions of $\mathrm{res}: \mathcal{O}_{L}(-1) \longrightarrow  \mathrm{Ext}_{C}^{1}(\mathcal{O}_{C}(H), F_{C})\otimes \mathcal{O}_{L}$ to fibers either have pairwise distinct image, or vanish at some $t\in L$.

  By the algorithm in \cite{Liu22} or a direct computation, we have $\mathrm{Ext}_{X_{t}}^{1}(E_{X_{t}}, E_{X_{t}}(-4H))=0$ for any smooth quartic surface $X_{t}$ with Picard rank 1. Hence $E_{X_{t}}|_{C}$ is simple, the map $\mathrm{res}_{t}$ is non-zero, hence the map $\mathrm{res}$ is also non-zero. If $\mathrm{res}$ vanishes at some $t\in L$, then the sheaf $R^{2}\pi_{*}\underline{Hom}_{\widetilde{\mathbb{P}^{3}}}(\mathcal{O}_{\widetilde{\mathbb{P}^{3}}}(H), \nu^{*}F(-Y))$ must have torsion. We exclude this by showing
  $$\mathrm{Ext}_{X_{t}}^{2}(\mathcal{O}_{X_{t}}(H), F|_{X_{t}}(-4H))\cong \mathrm{Hom}_{X_{t}}(F|_{X}(-4H), \mathcal{O}_{X_{t}}(H))^{*}$$
  has constant dimension for $t\in L$. Applying $\mathrm{Hom}_{\mathbb{P}^{3}}(F(-4H),-)$ to the exact sequence
  $$0 \longrightarrow \mathcal{O}_{\mathbb{P}^{3}}(-3H) \longrightarrow \mathcal{O}_{\mathbb{P}^{3}}(H) \longrightarrow \mathcal{O}_{X_{t}}(H) \longrightarrow 0, $$
  we get
  \begin{multline*}
    0 \longrightarrow \mathrm{Hom}_{\mathbb{P}^{3}}(F(-4H), \mathcal{O}_{\mathbb{P}^{3}}(-3H)) \longrightarrow \mathrm{Hom}_{\mathbb{P}^{3}}(F(-4H), \mathcal{O}_{\mathbb{P}^{3}}(H))\\
    \longrightarrow \mathrm{Hom}_{X_{t}}(F|_{X}(-4H), \mathcal{O}_{X_{t}}(H)) \longrightarrow \mathrm{Ext}^{1}_{\mathbb{P}^{3}}(F(-4H), \mathcal{O}_{\mathbb{P}^{3}}(-3H))=0.
  \end{multline*}
Hence $\mathrm{Ext}_{X_{t}}^{2}(\mathcal{O}_{X_{t}}(H), F|_{X_{t}}(-4H))$ has constant dimension, we must have that the map $\mathrm{res}: \mathcal{O}_{L}(-1) \longrightarrow  \mathrm{Ext}_{C}^{1}(\mathcal{O}_{C}(H), F_{C})\otimes \mathcal{O}_{L}$ has non-constant image.
\end{proof}

\begin{corollary}\label{nonisomorphic}
For distinct $t, t'\in L$, we have $E_{X_{t}}|_{C}\not\cong E_{X_{t'}}|_{C}$. 
\end{corollary}

\begin{proof}
By Proposition \ref{uniqueinclusion}, $E_{X}$ satisfies $\mathrm{Hom}(F|_{C}, E_{X})=\mathbb{C}$ for any $X\in L$. If $E_{X_{t}}|_{C}\cong E_{X_{t'}}|_{C}$ for any $t\neq t'$, by Lemma \ref{scalar}, the extension classes given by (\ref{EX}) are scalar multiples. By Proposition \ref{nonconstant}, this is not true.
\end{proof}

Combining all preparations, we now prove Theorem \ref{rank27}. 

\begin{proof}[Proof of Theorem \ref{rank27}]
  Suppose $E$ is an exceptional bundle on $\mathbb{P}^{3}$ of rank $27$ and degree $11$. Take distinct $t_{1}, t_{2}\in L$ so that $\mathrm{Pic}(X_{t_{i}})\cong \mathbb{Z}H, i=1,2$. By \cite{Zub90}, we have
  $$E|_{X_{t_{i}}}\cong E_{X_{t_{i}}}\in M_{X_{t_{i}}, H}(v), i=1,2. $$
  Restricting to $C$, we have
  $E_{X_{t_{1}}}|_{C} \cong E|_{C} \cong E_{X_{t_{2}}}|_{C}$.
 By Corollary \ref{nonisomorphic}, this is not true.
\end{proof}

\section*{References}

\bibliographystyle{alpha}
\renewcommand{\section}[2]{} 
\bibliography{reference}

\end{document}